\documentclass[11pt]{article}

\usepackage[english]{babel} 
\usepackage{amsmath,amssymb,amsfonts,amsthm}
\usepackage{mathrsfs, esint, dsfont}
\usepackage{moreverb,rotating,graphics,float}
\usepackage{caption, subcaption}
\usepackage[colorlinks, linkcolor={blue}, citecolor={red}]{hyperref}
\usepackage{framed,fancybox}
\usepackage{enumerate}
\usepackage{csquotes}
\usepackage{slashed}
\usepackage{physics}
\usepackage[left=2cm, right=2cm, top=2cm, bottom=2cm]{geometry}
\usepackage{graphicx}
\usepackage{indentfirst}
\usepackage{enumitem}
\usepackage{mathtools}
\graphicspath{ {./images/} }
\usepackage{tikz-cd}
\usepackage{MnSymbol}

\makeatletter
\renewcommand\Mn@Text@With@MathVersion[1]{%
  \textnormal{\ifx\Mn@Bold\math@version\bfseries\fi#1}%
}
\makeatother

\numberwithin{equation}{section} 
\theoremstyle{plain} 
\newtheorem{theorem}{Theorem}[section]
\newtheorem{proposition}[theorem]{Proposition}
\newtheorem{lemma}[theorem]{Lemma}
\newtheorem{corollary}[theorem]{Corollary}
\newtheorem{notation}[theorem]{Notation}

\theoremstyle{remark} 
\newtheorem{remark}[theorem]{Remark}

\theoremstyle{definition} 
\newtheorem{definition}[theorem]{Definition}

\newtheorem{example}[theorem]{Example}

\newcommand{\ph}{\varphi}

\newcommand{\pairing}[2]{\langle #1\mathbin{|}#2\rangle} 
\newcommand{\Pairing}[2]{\llangle #1\mathbin{|}#2\rrangle} 
\newcommand{\product}[3]{\llangle #1\mathbin{|}#2\mathbin{|}#3\rrangle}

\newcommand{\RA}[1]{\overrightarrow{#1}}
\newcommand{\LA}[1]{\overleftarrow{#1}}
\newcommand{\LRA}[1]{\overleftrightarrow{#1}}

\def\cA{{\mathcal A}}

\def\cM{{\mathcal X}}

\def\cX{{\mathcal X}}
\def\cY{{\mathcal Y}}
\def\cZ{{\mathcal Z}}

\def\C{{\mathbb C}}

\def\RR{{\mathbb R}}
\def\SS{{\mathbb S}}
\def\T{{\mathbb T}}

\def\fg{{\mathfrak g}}

\DeclareMathOperator{\IM}{Im}
\DeclareMathOperator{\DER}{Der}
\DeclareMathOperator{\HOM}{Hom}

\newcommand{\A}{\mathcal{A}}  

\newcommand{\D}{\mathcal{D}}  
\renewcommand{\d}{\mathrm{d}_\Psi} 
\newcommand{\dee}{\mathrm{d}} 
\renewcommand{\H}{\mathcal{H}}  
\newcommand{\ox}{\otimes}     
\newcommand{\Sf}{\mathbb{S}}  
\newcommand{\X}{\mathcal{X}}  
 
\newcommand{\Y}{\mathcal{Y}} 
\renewcommand{\.}{\cdot}      
\newcommand{\alphar}{\overrightarrow{\alpha}}
\newcommand{\alphal}{\overleftarrow{\alpha}}
\newcommand{\nablar}{\overrightarrow{\nabla}} 
\newcommand{\nablal}{\overleftarrow{\nabla}} 
\newcommand{\sharpr}{\overrightarrow{\sharp}}
\newcommand{\sharpl}{\overleftarrow{\sharp}}
\newcommand{\flatr}{\overrightarrow{\flat}}
\newcommand{\flatl}{\overleftarrow{\flat}}

\title{Comparison of Levi-Civita connections in noncommutative geometry}

\author{Alexander Flamant\dag \thanks{email: \texttt{alexander.flamant@ens.psl.eu},
\texttt{b.mesland@math.leidenuniv.nl}, \texttt{renniea@uow.edu.au}
}, 
Bram Mesland\S, Adam Rennie\ddag
\\[3pt]
\dag Ecole Normale Superieure PSL, Paris, France
\\[3pt]
\S Mathematisch Instituut, Universiteit Leiden, Netherlands
\\[3pt]
\ddag School of Mathematics and Applied Statistics, University of Wollongong\\
Wollongong, Australia}

\begin{document}

\maketitle

\begin{abstract}
We compare the constructions of Levi-Civita connections for noncommutative algebras developed in  \cite{AWcurvature,Gosw&alCentredBimod,MRLC}. The assumptions in these various constructions differ, but when they are all defined, we provide direct translations between them. An essential assumption is that the (indefinite) Hermitian inner product on differential forms/vector fields provides an isomorphism with the module dual. By exploiting our translations and clarifying the simplifications that occur for centred bimodules, we extend the existence results for Hermitian torsion-free connections in  \cite{AWcurvature,Gosw&alCentredBimod}.
\end{abstract}

%
%

\section{Introduction}

%
%
%
%
%

In the last decade, there have been several approaches to defining Levi-Civita connections, and so a curvature tensor, in noncommutative geometry, \cite{AWcurvature,BMbook,Gosw&alCentredBimod,BGJ1,MRLC,Rosenberg}. All use the algebraic definition of connections on modules, but then different starting points and assumptions enter to prove existence and/or uniqueness. Our aim is to clarify the relationships between the various existence and uniqueness arguments. By providing a consistent language to translate between the different approaches, we extend the existence results for Hermitian torsion-free connections in  \cite{AWcurvature,Gosw&alCentredBimod}.

The first difference to address is whether the Levi-Civita connection is defined on vector fields/derivations or on differential forms. The vector field/derivation approach is taken in \cite{AWcurvature,AW2,Rosenberg} in the setting of $\theta$-deformations of free torus actions. 
The differential form approach fits well with the algebraic machinary of differential calculi described in \cite{BMbook}, and was the approach taken in \cite{Gosw&alCentredBimod,BGJ1,MRLC}. By being able to relate the two approaches, we recover much of the language 
and flexibility of differential geometry in the noncommutative setting.

An existence and uniqueness proof for the Levi-Civita connection when the one-forms belong to a class of centred bimodules (satisfying some side conditions including finite projectivity) was given in \cite{Gosw&alCentredBimod,BGJ1}, with $\theta$-deformations of free torus actions as a main example. In \cite{MRLC} the centredness of the one-forms was removed, with similar algebraic side conditions to those in \cite{Gosw&alCentredBimod,BGJ1}. The general framework of \cite{AWcurvature,AW2} yields a uniqueness statement for metric compatible torsion-free affine connections, but existence is only shown in examples, which crucially satisfy (up to duality) the stronger assumptions of \cite{Gosw&alCentredBimod,BGJ1}. 
An additional feature of \cite{AWcurvature,AW2,Gosw&alCentredBimod,BGJ1} is the introduction of non-degenerate indefinite signature Hermitian inner products (on vector fields and one-forms respectively) to incorporate pseudo-Riemannian metrics. A different, possibly more restrictive, approach to indefinite metrics appears in \cite{MRLC}.

The setting in which all of the above approaches make sense and can be compared is that of \emph{centred bimodules}.
To describe our results, recall that a centred bimodule $\cX$ over an algebra $\A$ is generated as an $\A$-bimodule by central elements $x_j\in\cX$, which satisfy $x_ja=ax_j$ for all $a\in\A$.

The special role of centredness for modules of one-forms $\Omega^1_\dee$  is due to the existence, proved by Skeide \cite{Skeide}, of a bimodule map $\sigma:(\Omega^1_\dee)^{\ox2}\to (\Omega^1_\dee)^{\ox2}$ satisfying $\sigma(\omega\ox\eta)=\eta\ox\omega$ 
when at least one of  $\omega,\eta\in \Omega^1_\dee$ is central. Thus we have a natural idempotent $\Psi=\frac{1}{2}(1+\sigma)$ on the two-tensors.

The remaining differences revolve around the assumptions on the inner product on vector fields or one-forms. In \cite{AWcurvature,AW2} the inner product on vector fields is assumed to yield an injective map $\X\ni X\mapsto (Y\mapsto \pairing{X}{Y})\in \RA{\HOM}_\A(\X,\A)$, and we call this weak non-degeneracy. In \cite{Gosw&alCentredBimod,BGJ1} the same map is assumed to be an isomorphism, and we call this strong non-degeneracy. Very recently, \cite{AH25} has obtained existence results using strong non-degeneracy and various symmetry constraints on the inner product.

In \cite{MRLC} it is assumed that there is a pre-$C^*$-inner product, that is, positive definite and therefore weakly non-degenerate, and that the algebra $\A$ is local (spectrally invariant and dense) inside a $C^*$-algebra. These analytic assumptions were used in two places: to guarantee the existence of frames for the one-forms\footnote{a particular kind of generating set, \cite{FL02,MRLC}}; and to obtain existence of a limit defining  an Hermitian torsion-free connection on one-forms. The existence of a frame implies strong non-degeneracy of the (in this case positive definite) inner product.

In Appendix \ref{sec:fgpipdb} we show that strongly non-degenerate inner products guarantee the existence of a suitable generalisation for frames. In Section \ref{subsec:exist} we show that for centred bimodules with strongly non-degenerate inner product the existence of an Hermitian torsion-free connection on one-forms follows purely algebraically. We stress that the latter argument does not hold in the generality of non-centred bimodules. Thus, for centred bimodules, the analytic assumption of locality in \cite{MRLC} can be effectively replaced by the algebraic assumption of strong non-degeneracy. Nevertheless, establishing strong non-degeneracy in practice may well require analytic, or other, data. 

To relate our results for differential forms to the results of \cite{AWcurvature,AW2}, we show that a centred Hermitian (in particular strongly non-degenerate) differential calculus 
gives rise to a Lie algebra of derivations such that the left and right dual bimodules define left and right pseudo-Riemannian calculi in the sense of  \cite{AWcurvature}.

We also extend the existence result of \cite{Gosw&alCentredBimod} by combining the methods of \cite{Gosw&alCentredBimod,BGJ1} with those of \cite{MRLC} in Theorem \ref{thm: existence and uniqueness on centred bimods}. The bimodules considered in \cite{Gosw&alCentredBimod,BGJ1} satisfy a stronger condition than centredness (see condition (ii) of \cite[Theorem 4.1]{Gosw&alCentredBimod}), whereas the proof given here uses only centredness. Using the aforementioned duality results, we derive an existence statement for metric compatible torsion-free affine connections on vector fields, in the sense of \cite{AWcurvature, AW2}, when the inner product is strongly non-degenerate, in Corollary \ref{cor:existence-on-vector-fields}.

Section \ref{sec:back-ground} reviews the definitions and setup of \cite{AWcurvature,MRLC}. Section \ref{sec:vec-form} relates the formalisms of vector fields and differential forms. Our results relating the differential form  connections to the vector field affine connections of \cite{AWcurvature} are in Section \ref{sec:LC}, as are our existence and uniqueness statements. Appendix \ref{sec:fgpipdb} presents the results on strongly non-degenerate inner products on suitably finite modules, and are of independent interest.

For simplicity, we work with unital complex $*$-algebras throughout, and refer to \cite[Section 7.1]{MRLC} for methods to remove the unitality assumption. \\

\textbf{Acknowledgements} AR thanks Universiteit Leiden for hospitality in 2024. AF also thanks Universiteit Leiden for hosting him as an intern in 2024, as part of his master's program at Ecole Normale Superieure PSL. The authors thank J. Arnlind, J. Bhowmick and D. Goswami for feedback on a preliminary version, and thank J. Arnlind for sharing the preprint \cite{AH25}.

\section{Background}
\label{sec:back-ground}

We begin by recalling the relevant parts of the formalism of \cite{AWcurvature} and \cite{MRLC}.

\subsection{Pseudo-Riemannian calculi on noncommutative vector fields}

In this section we begin by summarising the framework and results developed in \cite{AWcurvature} to extend (pseudo)-Riemannian geometry to the noncommutative context. To define and study connections for noncommutative algebras, \cite{AWcurvature} uses derivations as generalisations of vector fields on a smooth manifold.
Indeed,  smooth vector fields on a manifold $M$ are in one-to-one correspondence with derivations on $C^\infty(M)$, i.e. maps $\partial : C^\infty(M) \rightarrow C^\infty(M)$ such that $\partial(fg) = \partial(f) g + f \partial(g)$ for any two functions $f, g \in C^\infty(M)$. 
Vector fields then form both a Lie algebra (using composition of derivations) and a $C^\infty(M)$-module by right multiplication, as sections of a vector bundle. To generalise this construction to the noncommutative framework,
we recall the definitions and terminology from \cite{AWcurvature}.

\begin{definition}
\label{defn:AW-basics}
Let $\cA$ be a unital $*$-algebra.  Let  $\DER(\cA)$ be the set of derivations on $\cA$, a (complex) Lie algebra, but not in general an $\cA$-module. 
We assume the existence of a (right) $\cA$-module $\cX(\cA)$, a (complex) Lie subalgebra $\fg_\C$ of $\DER(\cA)$ and of a $\C$-linear map
\begin{equation}
	\ph : \fg_\C  \rightarrow \cX(\cA).
\label{eq:lie-alg-map}
\end{equation}
We call $\cX(\A)$ the module of (right) noncommutative vector fields. 
We denote by $\fg$ the real Lie algebra of Hermitian derivations in $\fg_\C$, i.e. those satisfying $\partial(a^*)^* = \partial(a)$ for all $a \in \cA$.
\end{definition}

Using this formalism we can import classical definitions from Riemannian geometry. First we need a notion of a metric on vector fields. We will need to consider both right and left $\A$-modules, so we adapt definitions from \cite{AWcurvature} in the obvious manner to allow both cases. The map $\ph$ can be defined in the same way for left or right modules.

\begin{definition}
\label{def:right-metric}
	We say that $(\cX(\cA), \bra{\cdot}\ket{\cdot}_\cA)$ is a \emph{right metric $\cA$-module} if $\cX(\A)$ is a right $\A$-module and $\bra{\cdot}\ket{\cdot}_\cA : \cX(\cA) \times \cX(\cA) \rightarrow \cA$ is a map such that for all $X,Y,Z\in\cX(\A)$ and $a\in\A$ we have
	\[
	\begin{split}
		& \pairing{X}{Y + Z}_\cA = \pairing{X}{Y}_\cA + \pairing{X}{Z}_\cA \\
		& \pairing{X}{Ya}_\cA = \pairing{X}{Y}_\cA a \\
		& \pairing{X}{Y}_\cA^* = \pairing{Y}{X}_\cA
	\end{split}
	\]
	and the inner product is weakly non-degenerate, meaning
	\begin{equation}
	\pairing{X}{Y}_\cA = 0 \mbox{ for all }Y \in \cX(\cA)\mbox{ implies }X = 0.
	\label{eq:weak-nd}
	\end{equation} 
	Similarly a \emph{left metric $\cA$-module} is a left $\A$-module $\cY(\A)$ endowed with a map $\prescript{}{\A}{\bra{\cdot}\ket{\cdot}}: \cY(\cA) \times \cY(\cA) \rightarrow \cA$ such that
	\[
	\begin{split}
		& \prescript{}{\A}{\pairing{X}{Y + Z}} = \prescript{}{\A}{\pairing{X}{Y}} + \prescript{}{\A}{\pairing{X}{Z}} \\
		& \prescript{}{\A}{\pairing{aX}{Y}} = a \prescript{}{\A}{\pairing{X}{Y}} \\
		& \prescript{}{\A}{\pairing{X}{Y}}^* = \prescript{}{\A}{\pairing{Y}{X}}
	\end{split}
	\]
    and  the inner product is weakly non-degenerate, meaning
    \begin{equation*}
    \prescript{}{\A}{\pairing{X}{Y}} = 0 \mbox{ for all }Y \in \cY(\cA)\mbox{ implies }X = 0.
    \end{equation*} 
    We omit subscripts $\cA$ when there is no ambiguity.
\end{definition}

\begin{remark}
    A pre-$C^*$-module over a pre-$C^*$-algebra satisfies  $\pairing{X}{X}_\A\geq0$ for all $X$, and so satisfies Definition \ref{def:right-metric}, but not conversely. The reason for the extra freedom allowed by \eqref{eq:weak-nd} is to allow for indefinite real inner products. That said, we will need to replace \eqref{eq:weak-nd} by a stronger non-degeneracy condition in the sequel to obtain our main results.
\end{remark}

There are three pieces of initial data for the approach of \cite{AWcurvature}: the Lie algebra $\fg$, the module $\cX(\cA)$ (or $\cY(\cA)$) and the metric $\bra{\cdot}\ket{\cdot}_\cA$ (or $\prescript{}{\A}{\bra{\cdot}\ket{\cdot}}$). To ensure the three elements are compatible, in particular so that we do not have $\fg = 0$, we require further conditions.

\begin{definition}
\label{defn:metric-calc}
	Denoting the pair $(\fg, \ph)$ by $\fg_\ph$, then the triple $(\cX(\cA), \bra{\cdot}\ket{\cdot}_\cA, \fg_\ph)$ (resp. $(\cY(\cA), \prescript{}{\A}{\bra{\cdot}\ket{\cdot}}, \fg_\ph)$) is called a \emph{right (resp. left) metric calculus} over $\cX(\cA)$ (resp. $\cY(\A)$) if
	 the image $\ph(\fg)$ generates $\cX(\cA)$ (resp. $\cY(\A)$) as a right (resp. left) $\cA$-module. It is called a \emph{right (resp. left) real metric calculus} if in addition
		\[
		 \pairing{X}{Y}^* = \pairing{X}{Y}\quad \mbox{for all }X, Y \in \ph(\fg).
		 \]
\end{definition}




\subsection{Modules of differential forms}
Differential forms over noncommutative algebras have appeared in many ways, from directly geometric to cohomological  applications. Our approach, and needs, are closest to those in \cite{BMbook,Landi}, and follow the notation and conventions of \cite{MRLC}.

In order to state our strong non-degeneracy condition on inner products, and later to dualise from differential forms to vector fields, we recall the definitions of module duals, as well as the centre of a bimodule.

\begin{definition}
\label{def:hom-centre}
Let $\cY$ be a right $\A$-module. Then $\RA{\HOM}_\A(\cY, \A)$ is the set of right module homomorphims $\cY\to\A$, which is a left $\A$-module. Similarly if $\cY$ is a left $\A$-module, then $\LA{\HOM}_\A(\cY, \A)$ is the set of left module homomorphims $\cY\to\A$, which is a right $\A$-module, and if $\cY$ is an $\A$-bimodule then $\LRA{\HOM}_\A(\cY, \A)$ is the set of bimodule maps $\cY\to\A$, and 
$\mathcal{Z}(\cY)=\{y\in\cY:\,ay=ya\ \mbox{ for all }a\in\A\}$ is the centre of $\cY$. Finally, an $\A$-bimodule $\cY$ is \emph{centred} if $\mathcal{Z}(\cY)$ generates $\cY$ as an $\A$-bimodule (or equivalently as a right or left module).
\end{definition}

Much of our discussion of modules of forms and vector fields is simplified by the notion of a $\dag$-bimodule.

\begin{definition}
\label{defn:stern-bimod}
A \emph{$\dag$-bimodule} over the $*$-algebra $\A$ is an $\A$-bimodule $\mathcal{X}$ that is equipped with  an antilinear involution $\dag:\mathcal{X}\to\mathcal{X}$ such that $(axb)^{\dag}=b^{*}x^{\dag}a^{*}$. Given a $\dag$-bimodule $\X$, a $\dag$-\emph{bimodule derivation} is a bimodule derivation $\mathrm{d}:\A\to \X$ such that $\mathrm{d}(a^{*})=-\mathrm{d}(a)^{\dag}$.
\end{definition}
\begin{remark}
In \cite{MRLC}, $\dag$-bimodules were also required to be finite projective and with an inner product. We will only use $\dag$-bimodules in this context below.
\end{remark}
The universal one-forms associated to a unital $*$-algebra $\A$ are the kernel of the multiplication map
\[
\Omega^1_u(\A)=\ker(m:\A\ox\A\to\A)=\big\{\sum a_i\delta(b_i):\,a_i,b_i\in\A\big\}.
\] 
The differential $\delta:\A\to\Omega^1_u(\A)$ is given by $\delta(b)=1\ox b-b\ox1$.
The universal forms are a $\dag$-bimodule with $\dag(a\delta(b))=-\delta(b^*)a^*=-\delta(b^*a^*)+b^*\delta(a^*)$, $a,b\in\A$.

The universal feature of $\Omega^1_u(\A)$ is that whenever we have an $\A$-$\dag$-bimodule $M$, and a $\dag$-bimodule derivation ${\rm d}:\A\to M$, there exists a $\dag$-bimodule map $\pi:\Omega^1_u(\A)\to M$ such that $\pi\circ\delta(b)=\dee(b)$ for all $b\in M$. The data $(M,{\rm d})$ is called a first order $\dag$-calculus for $\A$. 
\begin{definition}\cite[Definitions 1.4 and 1.15]{BMbook} 
\label{defn:first-order}
A first order differential structure $(\Omega^1_\dee(\A),\dag)$ for the $*$-algebra $\A$ is a first order $\dag$-calculus $(\Omega^1_\dee(\A),{\rm d})$ for $\A$ such that  $\Omega^1_\dee(\A)=\pi(\Omega^1_u(\A))$. 
The first order differential structure $(\Omega^1_\dee(\A),\dag)$ is Hermitian if $\Omega^1_\dee(\A)$ is a finitely generated projective $\A$-module and carries a right $\A$-valued inner product $\pairing{\cdot}{\cdot}_\A$ satisfying, for all $\omega,\eta\in\Omega^{1}_{\dee}(\A)$ and $a\in\A$
\begin{enumerate}[noitemsep]
\item $\pairing{\omega}{\eta a}_\A=\pairing{\omega}{\eta}_\A a$;
\item $\pairing{\omega}{\eta}^{*}_\A=\pairing{\eta}{\omega}_\A$;
\item $\pairing{a\omega}{\eta}_\A=\pairing{\omega}{a^{*}\eta}_\A$;
\item  the inner product is strongly non-degenerate, meaning that the map $g:\omega\mapsto \left( \eta \mapsto \pairing{\omega^\dagger}{\eta}_\cA \right)$ is an isomorphism $\Omega^{1}_{\dee}(\A)\to \overrightarrow{\mathrm{Hom}}_{\A}(\Omega^{1}(\A),\A)$ of left $\A$-modules.
\end{enumerate} 
\end{definition}
\begin{remark}
In \cite{MRLC}, condition 3. of Definition \ref{defn:first-order} was inadvertently omitted, despite being required throughout.
\end{remark}
\begin{remark}
An Hermitian first order differential structure has a left inner product on one-forms as well, given by ${}_\A\pairing{\omega}{\rho}:=\pairing{\omega^\dag}{\rho^\dag}_\A$. Observe that we have not asked for the inner product to be positive definite, and to incorporate indefinite inner products, we have condition 4. This is the form of non-degeneracy used in \cite{Gosw&alCentredBimod,BGJ1} and which we require to relate the three approaches we consider, and is stronger than the non-degeneracy imposed in Definition \ref{def:right-metric} by \cite{AWcurvature,AW2}, which amounts to injectivity of the map $\Omega^{1}_{\dee}(\A)\to \overrightarrow{\mathrm{Hom}}_{\A}(\Omega^{1}(\A),\A)$. 
\end{remark}
\begin{remark}
In \cite[Definition 3.5]{AH25}, symmetric bilinear forms (such as $g$) satisfying condition 4 are referred as being \emph{invertible}. Such forms are also related to the notion of \emph{invertible quantum metric} of \cite[Definition 1.15]{BMbook}. See also Remark \ref{rem:algebraicsequences} in Appendix A.
\end{remark}
\begin{remark} In \cite{MRLC}, the $*$-algebra $\A$ is always assumed to be dense and spectral invariant in a $C^{*}$-algebra $A$. This assumption allows for a well-defined notion of positivity in $\A$. The inner products considered in \cite{MRLC} are positive definite in the sense of Hilbert $C^{*}$-modules. In this context, the inner product will satisfy condition 4, a fact implied by the existence of frames.
\end{remark}

\begin{example}
Given a spectral triple $(\A,\H,\D)$, see \cite{CPR} for instance, the commutators of algebra elements $a\in\A$ and the self-adjoint operator $\D$ are bounded. We can then define the module of one-forms 
\[
\Omega^{1}_{\D}(\A):=\textnormal{span}\left\{a[\D,b]:a,b\in\A\right\}\subset\mathbb{B}(\H).
\]
We obtain a first order differential calculus $\mathrm{d}:\A\to \Omega^{1}_{\D}(\A)$ by setting $\mathrm{d}(b):=[\D,b]$. This calculus carries an involution $(a[\D,b])^{\dag}:=[\D,b]^{*}a^{*}$ induced by the operator adjoint.
Given the extra data of a strongly non-degenerate inner product $\pairing{\cdot}{\cdot}_\A$ on $\Omega^1_\D(\A)$, we find that $(\Omega^{1}_{\D}(\A),\pairing{\cdot}{\cdot}_\A,\dag)$ is an Hermitian differential structure. 
\end{example}

We recall some constructions from \cite{Landi,MRLC} for $(\Omega^{1}_{\dee}(\A),\dag)$. Writing $T^{k}_{\dee}(\A):=\Omega_\dee^1(\A)^{\ox_\A k}$ and $T^*_\dee(\A)=\oplus_kT^k_\dee(\A)$, the universal differential
forms $\Omega^*_{u}(\A)$ admit a representation
\begin{align}
\pi_\dee:\Omega^k_{u}(\A)\to T^{k}_{\dee}(\A)\quad \pi_\dee(a_0\delta(a_1)\cdots\delta(a_k))&=a_0\dee(a_1)\ox\cdots\ox\dee(a_k).
\end{align}
The modules $T^k_\dee(\A)$ all carry right (and so left) inner products whenever $\Omega^1_\dee(\A)$ does. For example, the inner product on $T^2_\dee(\A)$ is defined on simple tensors $\omega\ox\rho,\eta\ox\tau$ by $\pairing{\omega\ox\rho}{\eta\ox\tau}_\A=\pairing{\rho}{\pairing{\omega}{\eta}_\A\,\tau}_\A$.

Typically $\pi_\dee$ is not a map of differential algebras, as $T^*_\dee(\A)$ is not a differential algebra, but $\pi_\dee$ is an $\A$-bilinear map of associative $*$-$\A$-algebras, \cite{Landi,MRLC}. The $\dag$-structure on 
$T^{*}_{\dee}(\A)$ is given by the $\dag$ on $\Omega^1_\dee(\A)$ and (well-defined for the balanced tensor products)
\[
(\omega_1\ox\omega_2\ox\cdots\ox\omega_k)^\dag:= 
\omega_k^\dag\ox\cdots\ox\omega_2^\dag\ox\omega_1^\dag.
\]
 
The maps $\pi_\dee:\Omega^{*}_{u}(\A)\to T^*_\dee(\A)$ 
and $\delta:\Omega^{k}_{u}(\A)\to \Omega^{k+1}_{u}(\A)$ 
are typically not compatible  in the sense that $\delta$ need not map $\ker \pi_\dee$ to itself. Thus in general,  $T^{*}_{\dee}(\A)$ cannot be made into a differential algebra.
The issue to address is that
there are universal
forms $\omega\in \Omega^n_u(\A)$ for which $\pi_\dee(\omega)=0$
but $\pi_\dee(\delta(\omega))\neq0$. The latter are known as {\em junk tensors},  \cite[Chapter VI]{BRB}. 
We denote the $\A$-bimodules of junk tensors by
\begin{align*}
JT^k_\dee(\A)=\{\pi_\dee(\delta(\omega)):\,\pi_\dee(\omega)=0\}.
\end{align*}
Observe that the junk submodule only depends on the representation of the universal forms.

\begin{definition}
\label{ass:fgp-metric-junk}
A second order differential structure  $(\Omega^{1}_{\dee},\dag,\Psi)$ for a $*$-algebra $\A$ is a first order differential structure $(\Omega^{1}_{\dee}(\A),\dag)$ together with an idempotent $\Psi=\Psi^2:T^{2}_{\dee}\to T^{2}_{\dee}$ satisfying $\Psi\circ\dag=\dag\circ\Psi$ and $JT^2_\dee(\A)\subset{\rm Im}(\Psi)$. A second order differential structure is Hermitian if $(\Omega^{1}_{\dee}(\A),\dag)$ is an Hermitian first order structure with right inner product $\pairing{\cdot}{\cdot}_\A$, such that $\Psi=\Psi^{2}=\Psi^{*}$ is a projection. 
\end{definition}

A second order differential structure admits an exterior derivative $\d:\Omega^1_\dee(\A)\to T^2_\dee(\A)$ 
via 
\begin{equation}
\label{eq: second-order-diff}
\d(\rho)=(1-\Psi)\circ \pi_\dee\circ\delta\circ \pi_\dee^{-1}(\rho).
\end{equation} 
The differential satisfies $\d(\dee(a)) = 0$ for all $a \in \cA$ and
\[
    \dee_\Psi(\omega^\dag) = \dee_\Psi(\omega)^\dag, \qquad \d(a \omega b) = (1 - \Psi)(\dee a \ox \omega b) + a (\d \omega) b - (1 - \Psi)(a \omega \ox \dee b).
\]
The differential allows us to define curvature for modules, and formulate torsion for connections on one-forms. See \cite[Section 3]{MRLC} for details.

\section{Relating the vector field and differential form formalisms}
\label{sec:vec-form}

\subsection{Dualising noncommutative differential forms}

We will start by comparing Hermitian first and second order differential structures of \cite{MRLC} and the metric calculus of \cite{AWcurvature}. Later we will consider connections and compare the two constructions of Levi-Civita connections, and also relate them to the approach of \cite{Gosw&alCentredBimod,BGJ1}. Defining vector fields as (suitable submodules of) duals of one-forms appears in \cite{BGL} in their discussion of Koszul formulae. Indeed, most of our dualising formulae mirror those of \cite{BGL}, as they must, but as we work in the Hermitian setting, rather than bilinear, and with somewhat different hypotheses, we start from scratch. That said, most of our results in this section have analogues in \cite{BGL}.

\begin{definition}
\label{defn:vec-fields}
Let $\cA$ be a unital $*$-algebra and $(\Omega_{\dd}^1(\cA), \bra{\cdot}\ket{\cdot}_\cA, \dagger)$ be an Hermitian first-order differential structure over $\cA$ in the sense of Definition \ref{defn:first-order}.

We define right-linear, left-linear and bilinear vector fields to be respectively the right-$\cA$-linear, left-$\cA$-linear and bi-$\cA$-linear dual of $\Omega_{\dd}^1(\cA)$. We denote these
\begin{align*}
    & \RA{\cX}(\A) := \RA{\HOM}_\A(\Omega_{\dd}^1(\A), \A) \\
    & \LA{\cX}(\A) := \LA{\HOM}_\A(\Omega_{\dd}^1(\A), \A) \\
    & \LRA{\cX}(\A) := \LRA{\HOM}_\A(\Omega_{\dd}^1(\A), \A).
\end{align*}
\end{definition}

\noindent The inner product on $\Omega_{\dd}^1(\cA)$ induces $\C$-linear maps 
 \[
\begin{split}
	&  \sharpr : \ \Omega_{\dd}^1(\cA) \to \RA{\cX}(\cA) 
	  \quad \qquad \omega \mapsto \left( \eta \mapsto \pairing{\omega^\dagger}{\eta}_\cA \right) \\
	&  \sharpl : \ \Omega_{\dd}^1(\cA) \to \LA{\cX}(\cA) 
	 \quad \qquad \omega \mapsto \left( \eta \mapsto \pairing{\eta^\dagger}{\omega}_\cA \right),\quad \omega,\eta\in\Omega^1_\dee(\A).
\end{split}
\]

We write the $\C$-bilinear pairing of a right-linear vector field $X$ (left-linear vector field $Y$) with a differential one-form $\omega$ as 
\[
\Pairing{X}{\omega} := X(\omega) \in \cA \quad\mbox{and}\quad \Pairing{\omega}{Y} := Y(\omega) \in \cA.
\] 
Right-linear (resp. left-linear) vector fields form a left (resp. right) $\cA$-module and the pairings satisfy
\[
    \Pairing{aX}{\omega b} = a\Pairing{X}{\omega} b \quad\mbox{and}\quad \Pairing{a\omega}{Yb} = a \Pairing{\omega}{Y}b,
\]
for $X \in \LA{\cX}(\cA)$, $Y \in \RA{\cX}(\cA)$, $\omega \in \Omega_{\dd}^1(\cA)$ and $a,b \in \cA$.

\begin{definition}\label{def: product vector fields}
    Let $X$ and $Y$ be respectively a right- and a left-linear vector field. We define their product to be the $\C$-linear map $X\cdot Y : T^2_\dee(\cA) \rightarrow \cA$ given on simple tensors by
    \[
        (X \. Y)(\omega \ox \eta) =  \product{X}{\omega \ox \eta}{Y} := \Pairing{X}{\omega} \Pairing{\eta}{Y}.
    \]
\end{definition}

\begin{remark}
	The product of vector fields is well-defined on balanced tensors since for $X$ right-linear and $Y$ left-linear, 
	\[
	\Pairing{X}{\omega a} \Pairing{\eta}{Y} = \Pairing{X}{\omega} a \Pairing{\eta}{Y} = \Pairing{X}{\omega} \Pairing{a \eta}{Y}
	\] 
	so $\product{X}{\omega a \ox \eta}{Y} = \product{X}{\omega \ox a \eta}{Y}$. 
\end{remark}

We also consider the $\RR$-linear isomorphism $\dagger : \LA{\cX}(\cA) \rightarrow \RA{\cX}(\cA)$ defined using duality by
\begin{equation}
\begin{split}
    \Pairing{X^\dagger}{\omega} := \Pairing{\omega^\dagger}{X}^{*}.
\end{split}    
    \label{eq:pair-dag}
\end{equation}
By abuse of notation we also write $\dagger$ for its inverse $\dagger^{-1} : \RA{\cX}(\cA) \rightarrow \LA{\cX}(\cA)$.

\begin{proposition}
\label{prop:musicalisos}
    The musical maps $\sharpr, \sharpl$ are isomorphisms of left- and right-$\cA$-modules respectively, that are compatible with the dagger maps on forms and vector fields $\dagger \circ \RA{\sharp} = \LA{\sharp} \circ \dagger$. We denote by $\flatr, \flatl$ their inverses.
\end{proposition}

\begin{proof}
    Definition \ref{defn:first-order} immediately implies that $\sharpr, \sharpl$ are left- and right-$\cA$-linear maps respectively and that $\sharpr$ is an isomorphism of right modules. To show $\dag$-compatibility, let $\omega, \eta \in \Omega_{\dd}^1(\cA)$. Then using the definition of the $\dagger$-structure \eqref{eq:pair-dag} at the last equality
	\[
	\Pairing{(\omega^\dagger)^{\RA{\sharp}}}{\eta} = \pairing{\omega}{\eta}_\cA = \pairing{\eta}{\omega}_\cA^* = \Pairing{\eta^\dagger}{\omega^{\LA{\sharp}}}^{*} = \Pairing{(\omega^{\LA{\sharp}})^\dagger}{\eta}.
	\]
Hence also $\sharpl = \dag \circ \sharpr \circ \dag$ is an isomorphism of left modules.
\end{proof}

The right (resp. left) inner product on $\Omega_{\dd}^1(\cA)$ induces an Hermitian right (resp. left) inner product on $\LA{\cX}(\cA)$ (resp. $\RA{\cX}(\cA)$) by
\begin{equation}
    \pairing{X}{Y}_\cA := \pairing{X^{\LA{\flat}}}{Y^{\LA{\flat}}}_\cA \qquad \prescript{}{\cA}{\pairing{U}{V}} := \prescript{}{\cA}{\pairing{U^{\RA{\flat}}}{V^{\RA{\flat}}}},
\label{eq:vecfield-ips}
\end{equation}
for $X, Y \in \LA{\cX}(\cA)$ and $U, V \in \RA{\cX}(\cA)$. These inner products on vector fields inherit strong non-degeneracy from the inner product on one-forms.


\begin{proposition}\label{prop: central forms and bivector fields}
	The musical isomorphisms restrict to $\C$-linear isomorphisms between the centre $\cZ(\Omega_{\dd}^1(\cA))$ of $\Omega_{\dd}^1(\cA)$ and the bilinear vector fields $\overleftrightarrow{\cX}(\cA)$:
	\[
	\begin{split}
		&  \overrightarrow{\sharp} : \ \cZ(\Omega_{\dd}^1(\cA)) \xrightarrow{\cong} \overleftrightarrow{\cX}(\cA) \quad\mbox{and}\quad
		  \overleftarrow{\sharp} : \ \cZ(\Omega_{\dd}^1(\cA)) \xrightarrow{\cong} \overleftrightarrow{\cX}(\cA).
	\end{split}
	\]
\end{proposition}

\begin{proof}
	We only prove this for the right isomorphisms (the left follows in a similar manner). First, observe that if $\omega \in \cZ(\Omega_{\dd}^1(\cA))$ is a central one-form then, for all $\eta \in \Omega_{\dd}^1(\cA)$ and $a \in \cA$,
	\[
		\Pairing{\omega^{\overrightarrow{\sharp}}}{a \eta} = \pairing{\omega^\dagger}{a \eta}_\cA = \pairing{a^* \omega^\dagger}{\eta}_\cA = \pairing{\omega^\dagger a^*}{\eta}_\cA = a \pairing{\omega^\dagger}{\eta}_\cA = a \Pairing{\omega^{\RA{\sharp}}}{\eta}
	\]
	so $\overrightarrow{\sharp}$ maps $\cZ(\Omega_{\dd}^1(\cA))$ into $\overrightarrow{\cX}(\cA) \cap \overleftarrow{\cX}(\cA) = \overleftrightarrow{\cX}(\cA)$.
	Conversely, if $\omega^{\overrightarrow{\sharp}} \in \LRA{\cX}(\cA)$ then, for all $\eta \in \Omega_{\dd}^1(\cA)$ and $a \in \cA$, using that $\RA{\sharp}$ is a left module map we have
	\[
		\Pairing{(a \omega)^{\RA{\sharp}}}{\eta} = a \Pairing{\omega^{\RA{\sharp}}}{\eta} = \Pairing{\omega^{\RA{\sharp}}}{a \eta} = \pairing{\omega^\dagger}{a \eta}_\cA = \pairing{(\omega a)^\dagger}{\eta}_\cA = \Pairing{(\omega a)^{\RA{\sharp}}}{\eta}
	\]
	hence, since $\RA{\sharp}$ is an isomorphism and the pairing is non-degenerate, $a \omega = \omega a$.
\end{proof}


\begin{remark}
\label{remarkable}
In general, the isomorphisms $\sharpr,\sharpl$ do not coincide on $\mathcal{Z}(\Omega^1_\dee(\A))$. This is because the two pairings that define $\sharpr,\sharpl$ are distinct and additional information is required to relate them. For if $\omega$ is a central one-form and $\eta$ is any one-form we have $\Pairing{\omega^{\sharpr}}{\eta}=\pairing{\omega^\dag}{\eta}_\A$ while $\Pairing{\eta}{\omega^{\sharpl}}=\pairing{\eta^\dag}{\omega}_\A$, and in general these are distinct elements of $\A$. 
\end{remark}

The following lemma is another expression of the duality of bilinear vector fields and central one-forms, and will be used repeatedly.

\begin{lemma}
\label{lem:bilinear-central}
If $X\in \LRA{\cX}(\cA)$ is a bilinear vector field and $\omega\in \mathcal{Z}(\Omega^1_\dee(\A))$ is a central one-form then $\Pairing{X}{\omega}\in\mathcal{Z}(\A)$ is a central element of $\A$.
\end{lemma}
\begin{proof}
With $X$ and $\omega$ as in the statement we use Proposition \ref{prop: central forms and bivector fields} to write $X=\eta^{\sharpr}$ with $\eta$ a central one-form. For arbitrary $a\in \A$ we have
\begin{align*}
a\Pairing{X}{\omega}&=a\Pairing{\eta^{\sharpr}}{\omega}=a\pairing{\eta^\dag}{\omega}_\A=\pairing{(a\eta)^\dag}{\omega}_\A=\pairing{(\eta a)^\dag}{\omega}_\A\\
&=\pairing{a^*\eta^\dag}{\omega}_\A=\pairing{\eta^\dag}{a\omega}_\A=\pairing{\eta^\dag}{\omega a}_\A=\Pairing{X}{\omega} a.\qedhere
\end{align*}
\end{proof}

\begin{proposition}
\label{prop:formstovectorfields}
Let $\cA$ be a unital $*$-algebra and $(\Omega_{\dd}^1(\cA), \bra{\cdot}\ket{\cdot}_\cA, \dagger)$ be an Hermitian first-order differential structure over $\cA$ in the sense of Definition \ref{defn:first-order}. Then the pair $\big( \LA{\cX}(\cA), \bra{\cdot}\ket{\cdot}_\cA \big)$ is a right metric $\A$-module. Symmetrically the pair $\big( \RA{\cX}(\cA), \prescript{}{\cA}{\bra{\.}\ket{\.}} \big)$ is a left metric $\A$-module.
\end{proposition}

\begin{proof}
    Linearity in the second argument and Hermitian symmetry are clear from the fact that $\bra{\cdot}\ket{\cdot}_\cA$ is an Hermitian inner product on $\Omega_{\dd}^1(\cA)$. Weak non-degeneracy follows from the fact that the inner product on one-forms 
    is strongly non-degenerate.
%
\end{proof}

As such, a single inner product over a bimodule of one-forms $\Omega_{\dd}^1(\cA)$ yields two separate inner product structures over vector fields, one for the right module $\LA{\cX}(\cA)$ and one for the left module $\RA{\cX}(\cA)$. These are related by the dagger map described above.

\subsection{Derivations and first-order differential calculi}

In this section, we show how noncommutative vector fields seen as duals of differential forms,  are naturally related to derivations on the algebra $\A$, extending the well-known fact that vector fields on a smooth manifold act as derivations on smooth functions.

We fix  a unital $*$-algebra $\A$, an Hermitian first-order differential structure $(\Omega_{\dd}^1(\cA), \bra{\cdot}\ket{\cdot}_\cA, \dagger)$,  and let $\DER(\cA)$ be the set of derivations of $\cA$ into itself.

\begin{proposition}\label{def: def_map_phi}
    Consider the $\C$-linear subspace of $\DER(\cA)$ defined by
    \[
        D_\C := \Big\{ \partial \in \DER(\cA)\ \Big|\ \forall a_i, b_i \in \cA,\ \sum\limits_i a_i \dd b_i = 0 \implies \sum\limits_i a_i \partial b_i = 0  \Big\}.
    \]
    We define a $\C$-linear map $\ph : D_\C \rightarrow \RA{\cX}(\cA)$ by
        \[
        \Pairing{\ph(\partial)}{\sum\limits_i a_i \dd b_i} :=  \sum\limits_i a_i \partial b_i.
    \]
    Then the image of $\ph$ is contained in $\LRA{\cX}(\A)$ and the map $\ph : D_\C \rightarrow \LRA{\cX}(\cA)$ is a $\C$-linear isomorphism.
\end{proposition}

\begin{proof} That $\varphi$ is well-defined follows from the definition of $D_\C$.
    We first show that $\ph$ takes its values in $\LRA{\cX}(\cA)$. 
    Indeed, suppose $\partial\in D_\C$. Using that $\dd$ is a derivation we compute, for $a, b, a_i, b_i \in \cA$,
    \[
        a \Big( \sum\limits_i a_i \dd b_i \Big) b = \sum\limits_i a a_i \dd( b_i b) - \sum\limits_i a a_i b_i \dd b.
    \]
    Hence
    \[
        \Pairing{\ph(\partial)}{a \Big( \sum\limits_i a_i \dd b_i \Big) b} = \sum\limits_i a a_i \partial( b_i b) - \sum\limits_i a a_i b_i \partial b = a \Big( \sum\limits_i a_i \partial b_i \Big) b=a\Pairing{\ph(\partial)}{\Big( \sum\limits_i a_i \dd b_i \Big)}b,
    \]
    so $\ph(\partial)$ is $\cA$-bilinear. To show that $\ph:D_\C\to \LRA{\cX}(\A)$ is an isomorphism, we construct the inverse. Let $X\in \LRA{\cX}(\cA)$ be a bilinear vector field. Define a derivation $\partial_X$ on $\cA$ by $\partial_X(a) := \Pairing{X}{\dd a}$, $a \in \cA$. Then, for all $a_i, b_i \in \cA$ such that $\sum_i a_i \dd b_i = 0$, we have by bilinearity of $X$
    \[
        \sum_i a_i \partial_X (b_i) = \Pairing{X}{\sum_i a_i \dd b_i} = 0.
    \]
    So $\partial_X \in D_\C$, 
    and the map $ \LRA{\cX}(\cA)\ni X\mapsto \big(\cA\ni a \mapsto \Pairing{X}{\dd a}\big)$ is the inverse of $\ph$.
\end{proof}

%
%

%
%
In the next section we will identify the key additional assumption which allows us to relate the identification of (real) vector fields as duals of differential forms and the identification with (Hermitian) derivations. We make the following definition.

\begin{definition}
    We say that a bilinear vector field $X \in \LRA{\cX}(\cA)$ is \emph{real} if $X^\dagger = -X$. We say that a derivation $\partial \in \DER(\cA)$ is \emph{Hermitian} if $\partial^\dagger = \partial$, where $\partial^\dagger(a) := \partial(a^*)^*$. We write the sets of real bilinear vector fields and Hermitian derivations as $\LRA{\cX}(\cA)^\dagger$ and $\DER(\cA)^\dagger$ respectively. We let $D := D_\C^\dagger = D_\C \cap \DER(\cA)^\dagger$ be the subset of Hermitian derivations in $D_\C$.
\end{definition}

%
%

\subsection{Centred bimodules of differential forms}

To make contact with the constructions of \cite{AWcurvature} summarised in Definition \ref{defn:metric-calc}, we need to impose further conditions on our modules of vector fields  and the map $\ph$ from derivations on $\cA$ to our module(s). 
In particular we need to check that:
\begin{itemize}
    \item $D$ is a real Lie algebra of derivations, i.e. it is stable under commutators;
    \item The image of $D$ by $\ph$ generates the module $\LA{\cX}(\cA)$ (or symmetrically $\RA{\cX}(\cA)$).
\end{itemize}
By Proposition \ref{def: def_map_phi}, the latter condition, at least for $D_\C$, reduces to checking that $\LRA{\cX}(\cA)$ generates $\LA{\cX}(\cA)$ as a right module (or equivalently $\RA{\cX}(\cA)$ as a left module, since the dagger map intertwines the two module structures). Using Proposition \ref{prop: central forms and bivector fields}, this is in turn the same as requiring that the central forms $\cZ(\Omega^1_\dee(\cA))$ generate $\Omega^1_\dee(\cA)$ (as a right, left or bimodule). Thus the second condition above is equivalent to the bimodule of differential forms being centred, as in Definition \ref{def:hom-centre}. We will see in Theorem \ref{thm: Lie algebra} that when $\Omega_\dee^1(\cA)$ is a centred bimodule and some mild assumptions hold, $D$ is moreover a real Lie subalgebra of $\DER(\cA)^\dag$ and plays the role of $\fg$ in \cite{AWcurvature}.

First we need some results on centred bimodules. Recall that a braiding on a $\dag$-$\cA$-bimodule $\cM$ is an invertible map $\sigma : \cM \ox_\cA \cM \rightarrow \cM \ox_\cA \cM$ such that $\sigma^{-1} \circ \dag = \dag \circ \sigma$. By Proposition \ref{prop: braiding} of the Appendix, on a centred $\dagger$-bimodule $\cM$ over $\cA$, there exists a unique braiding $\sigma^\mathrm{can}$ which satisfies $\sigma^{\mathrm{can}}(x \otimes y) = y \otimes x$ whenever $x$ or $y \in \cZ(\cM)$. This braiding is an involution.



\begin{remark}
    The canonical braiding $\sigma^\mathrm{can}$ on $\cM$ yields an idempotent $\Psi := \frac{1}{2}(1 + \sigma^\mathrm{can})$ which, for differential forms on a manifold, is the junk projection onto symmetric tensors \cite[Section 6.5]{MRLC}.
\end{remark}

\begin{definition}
\label{defn:metric-bilinear}
	We define the bilinear quantum metric to be the map $g : T^2_\dee(\cA) \rightarrow \cA$ given on the simple tensor $\omega\ox\eta\in T^2_\dee(\A)$ by $g(\omega \otimes \eta) := -\pairing{\omega^\dagger}{\eta}_\cA$. Observe that
	$g(a\omega \otimes \eta b)=ag(\omega \otimes \eta)b$ for $a,b\in\A$.
\end{definition}

\begin{remark}
    The bilinear quantum metric is well-defined on balanced tensors since for $a \in \cA$ we have $g(\omega a \ox \eta) = -\pairing{a^* \omega^\dag}{\eta} = -\pairing{\omega^\dag}{a \eta} = g(\omega \ox a \eta)$.
\end{remark}

\begin{remark}
    We can reformulate the musical isomorphisms in term of the bilinear quantum metric. For $\omega,\eta\in \Omega^1_\dee(\A)$ we have $\Pairing{\omega^{\RA{\sharp}}}{\eta} = -g(\omega \otimes \eta)$ and $\Pairing{\eta}{\omega^{\LA{\sharp}}} = -g(\eta \otimes \omega)$.
\end{remark}

The following is the core definition of this paper: it is the setting in which we will show that the constructions of \cite{AWcurvature}, \cite{Gosw&alCentredBimod}, \cite{MRLC} all make sense and coincide.

\begin{definition}
\label{def:canonicalcentred}
    We say that the Hermitian first-order differential structure $(\Omega_{\dd}^1(\cA), \bra{\cdot}\ket{\cdot}_\cA, \dagger)$ is a \emph{centred Hermitian differential calculus} if the following holds:
    \begin{enumerate}
        \item The bimodule $\Omega^1_\dee(\cA)$ is centred over $\cA$, i.e. is generated as an $\cA$-bimodule by its centre;
        
        \item The bilinear quantum metric on $\Omega^1_\dee(\cA)$ is invariant under the canonical braiding: $g \circ \sigma^\mathrm{can} = g$;
        
        \item The submodule $JT^2_\dee$ of junk forms satisfies $JT^2_\dee \subset \IM \Psi$ where $\Psi := \frac{1}{2}(1 + \sigma^\mathrm{can})$.
    \end{enumerate}
    We denote such a calculus by $(\Omega_{\dd}^1(\cA), \dagger, \bra{\cdot}\ket{\cdot},\sigma^\mathrm{can})$.
\end{definition}

\begin{remark}
    We will often, when there is no ambiguity, drop the superscript on the braiding $\sigma^\mathrm{can}$.
\end{remark}

\begin{proposition}
\label{prop: psi-proj}
    A centred Hermitian differential calculus $(\Omega_{\dd}^1(\cA), \dagger, \bra{\cdot}\ket{\cdot},\sigma^\mathrm{can})$ is an Hermitian second order differential structure in the sense of Definition \ref{ass:fgp-metric-junk}.
\end{proposition}

\begin{proof}
    The only thing to prove is that $\Psi$ is a projection, and since $\sigma^2={\rm Id}$ we immediately have $\Psi^2 = \Psi$. For self-adjointness it suffices to prove that $\sigma$ is self-adjoint on $T^2_\dee$. For this, it suffices to work with simple tensors $\omega\ox\rho,\eta\ox\tau$ of central one-forms. For these, repeated use of centrality yields
    \begin{align*}
    \pairing{\sigma(\omega\ox\rho)}{\eta\ox\tau}_\A
    &=\pairing{\rho\ox\omega}{\eta\ox\tau}_\A
    =\pairing{\omega}{\pairing{\rho}{\eta}_\A\,\tau}_\A
    =\pairing{\omega}{\tau}_\A\,\pairing{\rho}{\eta}_\A\\
    &=\pairing{\rho\pairing{\tau}{\omega}_\A}{\eta}_\A
    =\pairing{\pairing{\tau}{\omega}_\A\,\rho}{\eta}_\A
    =\pairing{\rho}{\pairing{\omega}{\tau}_\A\,\eta}_\A\\
    &=\pairing{\omega\ox\rho}{\tau\ox\eta}_\A
    =\pairing{\omega\ox\rho}{\sigma(\eta\ox\tau)}_\A.\qedhere
    \end{align*}
\end{proof}

\begin{proposition}
\label{cor:centred-dual}
Let  $(\Omega_{\dd}^1(\cA), \bra{\cdot}\ket{\cdot}_\cA, \dagger)$ be a centred Hermitian differential calculus. Then the musical isomorphisms $\sharpr, \sharpl$ coincide on the central one-forms $\mathcal{Z}(\Omega^1_\dee(\A))$. 
\end{proposition}
\begin{proof}
We recall from Remark \ref{remarkable} that if $\omega$ is a central one-form and $\eta$ is any one-form, the pairings defining the musical isomorphisms are 
\[
\Pairing{\omega^{\sharpr}}{\eta}=\pairing{\omega^\dag}{\eta}_\A=-g(\omega\ox\eta)\quad\mbox{and}\quad\Pairing{\eta}{\omega^{\sharpl}}=\pairing{\eta^\dag}{\omega}_\A=-g(\eta\ox\omega).
\] 
Since $g\circ \sigma=g$, the pairings coincide and so the musical isomorphisms $\sharpr,\sharpl$ coincide on $\mathcal{Z}(\Omega^1_\dee(\A))$. 
\end{proof}

\begin{proposition}
\label{prop:real-der}
The complex linear isomorphism $\ph:D_\C\to\LRA{\cX}(\A)$ restricts to a real linear isomorphism $\ph:D\to\LRA{\cX}(\A)^\dag$.
\end{proposition}
\begin{proof}
    Let $\partial\in D=D_\C\cap \DER(\A)^\dag$. Using \eqref{eq:pair-dag} we have
    \[
    \begin{split}
        \Pairing{\sum\limits_i a_i \dd b_i}{\ph(\partial)^\dagger} & = \Pairing{\ph(\partial)}{-\sum\limits_i \dd(b_i^*) a_i^*}^{*} \\
        & = \Big( -\Pairing{\ph(\partial)}{\sum\limits_i \dd(b_i^* a_i^*)} + \Pairing{\ph(\partial)}{\sum\limits_i b_i^* \dd(a_i^*)} \Big)^* \\
        & = \Big( -\sum\limits_i \partial(b_i^* a_i^*) + \sum\limits_i b_i^* \partial(a_i^*) \Big)^* \\
        & = - \sum\limits_i a_i \partial(b_i^*)^* = -\Pairing{\ph(\partial)}{\sum\limits_i a_i \dd b_i}= -\Pairing{\sum\limits_i a_i \dd b_i}{\ph(\partial)},
    \end{split}
    \]
    where we used the Leibniz rule for both $\dd$ and $\partial$, $\partial(b_i^*)^* = \partial(b_i)$, and the bilinearity of $\ph(\partial)$ at the last equality.     
     Hence we deduce that $\ph(\partial)^\dagger = -\ph(\partial)$. Conversely, if $\ph(\partial)$ is real then $\partial^\dag(a) = \partial(a^*)^* = \Pairing{\ph(\partial)}{\dee(a^*)}^* = -\Pairing{\ph(\partial)^\dag}{\dee a} = \partial(a)$.
\end{proof}

\begin{lemma}\label{lemma: antisym_vector_fields}
    Suppose $(\Omega_{\dd}^1(\cA), \dagger, \Psi, \bra{\cdot}\ket{\cdot},\sigma^\mathrm{can})$ is a centred Hermitian differential calculus. Let $X, Y$ be two bilinear vector fields. For any two-tensor $\alpha \in T^2_{\dd}(\cA)$,
    \[
        \product{X}{\alpha}{Y} - \product{Y}{\alpha}{X} = 2 \product{X}{(1 - \Psi) \alpha}{Y}.
    \]
\end{lemma}

\begin{proof}
    Since $\Omega_{\dd}^1(\cA)$ is centred, $T^2_{\dd}(\cA)$ is generated by elements of the form $\eta \otimes \rho$ with $\eta$ and $\rho$ central one-forms. Then  the definition of the pairing  and the centrality implied by Lemma \ref{lem:bilinear-central} yield
    \begin{align*}
        \product{X}{\eta \otimes \rho}{Y} - \product{Y}{\eta \otimes \rho}{X} &= 
        \Pairing{X}{\eta}\Pairing{\rho}{Y}-\Pairing{Y}{\eta}\Pairing{\rho}{X}\\
        &= \Pairing{X}{\eta}\Pairing{\rho}{Y}-\Pairing{X}{\rho}\Pairing{\eta}{Y}\\
        &=\Pairing{X\cdot Y}{\eta \otimes \rho - \rho \otimes \eta}.
    \end{align*}
    Since $\sigma$ is the flip on $\eta \otimes \rho$, we have
    \[
        \eta \otimes \rho - \rho \otimes \eta = (1 - \sigma) \eta \otimes \rho = 2 (1 - \Psi) \eta \otimes \rho.
    \]
    The general result then follows by $\cA$-bilinearity of $\Psi$.
\end{proof}

\begin{theorem}
\label{thm: Lie algebra}
    Let $(\Omega_{\dd}^1(\cA), \dagger, \Psi, \bra{\cdot}\ket{\cdot},\sigma^\mathrm{can})$ be a centred Hermitian differential calculus. Then $D_\C$ is a complex Lie algebra of derivations, i.e. $D_\C$ is stable under the commutator of derivations. Hence $D$ is a real Lie subalgebra of $\DER(\cA)^\dag$.
\end{theorem}

\begin{proof}
	Let $\partial_1, \partial_2 \in D_\C$. We need to show that their commutator $[\partial_1, \partial_2]$ is again in $D_\C$, i.e. that $\sum_i a_i [\partial_1, \partial_2] (b_i) = 0$ whenever $a_i, b_i \in \cA$ satisfy $\sum_i a_i \dd b_i = 0$. For such $a_i, b_i$,
	\begin{alignat*}{2}
		\sum\limits_i a_i [\partial_1, \partial_2](b_i) &= && \sum\limits_i a_i \partial_1 \partial_2 b_i - a_i \partial_2 \partial_1 b_i \\
		&= && \sum\limits_i \partial_1( a_i \partial_2 b_i ) - \partial_2( a_i \partial_1 b_i ) - (\partial_1 a_i) (\partial_2 b_i) + (\partial_2 a_i) (\partial_1 b_i) \\
		&= && \sum\limits_i \partial_1( a_i \partial_2 b_i ) - \partial_2( a_i \partial_1 b_i ) \\
		& && - \big[ \product{\ph(\partial_1)}{\dd a_i \otimes \dd b_i}{\ph(\partial_2)} - \product{\ph(\partial_2)}{\dd a_i \otimes \dd b_i}{\ph(\partial_1)} \big].
	\end{alignat*}
	Now, since $\partial_1, \partial_2 \in D_\C$ and $\sum_i a_i \dd b_i = 0$, the first two terms vanish. We are left only with the last term, which by Lemma \ref{lemma: antisym_vector_fields} reduces to
	\[
		\sum\limits_i a_i [\partial_1, \partial_2](b_i) 
		= -2 \sum\limits_i \product{\ph(\partial_1)}{(1-\Psi)(\dee a_i\ox \dee b_i)}{\ph(\partial_2)}
		= -2 \product{\ph(\partial_1)}{\dd_\Psi( \sum\limits_i a_i \dd b_i )}{\ph(\partial_2)} = 0.
	\]
	Hence $[\partial_1,\partial_2]\in D_\C$, and $D_\C$ is a complex Lie subalgebra of $\DER(\cA)$. By restricting to Hermitian derivations we also get that $D$ is a real Lie subalgebra of $\DER(\cA)^\dag$.
\end{proof}

Proposition \ref{prop: psi-proj} and Theorem \ref{thm: Lie algebra} show that centred Hermitian differential calculi are a natural class of noncommutative calculi where the constructions of \cite{MRLC}, \cite{Gosw&alCentredBimod} and \cite{AWcurvature} can be compared. 
We also see that, despite these constructions being well-defined for first-order calculi (of forms or vector fields), independently of second-order phenomena, we need the second order information encoded in $\Psi$ for them to be dual. As we shall show subsequently, the second-order structure also appears when extending the comparison to  Levi-Civita connections.

In order to complete our comparison at the level of differential calculi, it must now be shown that  dualising a centred Hermitian structure satisfies all the conditions of Definition \ref{defn:metric-calc}, making the dual a real metric calculus. 
%
%
%
%
%
%
By Proposition \ref{cor:centred-dual}, we may and shall
omit arrows on the musical isomorphisms when working with bilinear vector fields or central forms on centred Hermitian differential calculi.

\begin{corollary}\label{cor: symmetry inner product}
    Suppose $(\Omega_{\dd}^1(\cA), \dagger, \bra{\cdot}\ket{\cdot},\sigma^\mathrm{can})$ is a centred Hermitian differential calculus. Let $X, Y \in \LRA{\cX}(\cA)$. The following hold:
    \begin{enumerate}
    \itemsep-3pt
        \item $( X^\flat )^\dagger = ( X^\dagger )^\flat$;
        \item $\pairing{X}{Y}_\cA = \pairing{Y^\dagger}{X^\dagger}_\cA$;
        \item If $X, Y$ are real vector fields, $\pairing{X}{Y}^*_\cA = \pairing{X}{Y}_\cA$.
    \end{enumerate}
\end{corollary}

\begin{proof}
    The first point follows from 
    Proposition \ref{cor:centred-dual} and the fact that $\dag \circ \LA{\flat} = \RA{\flat} \circ \dag$. The second is a consequence of the first, since
    \[
        \pairing{X}{Y}_\cA = \pairing{X^\flat}{Y^\flat}_\cA = \pairing{(Y^\flat)^\dagger}{(X^\flat)^\dagger}_\cA = \pairing{(Y^\dagger)^\flat}{(X^\dagger)^\flat}_\cA = \pairing{Y^\dagger}{X^\dagger}_\cA,
    \]
    where for the second equality we have used that $g(\omega \otimes \eta) = g(\eta \otimes\omega)$ for central forms $\omega, \eta$. The third point follows from the second, since a real vector field $X$ satisfies $X^\dagger = -X$ and the inner product is Hermitian.
\end{proof}

\begin{lemma}
\label{lem:der-gen}
    Let $(\Omega_{\dd}^1(\cA), \dagger, \Psi, \bra{\cdot}\ket{\cdot},\sigma^\mathrm{can})$ be a centred Hermitian differential calculus. The set of bilinear vector fields $\LRA{\cX}(\cA)$ generates $\LA{\cX}(\cA)$ (resp. $\RA{\cX}(\cA)$) as a right (resp. left) $\cA$-module.
\end{lemma}

\begin{proof}
    Since $\LA{\sharp}$ and $\RA{\sharp}$ are respectively a right and left module map, this follows from the isomorphism $\cZ(\Omega_{\dd}^1(\cA)) \cong \LRA{\cX}(\cA)$ and the fact that $\Omega_{\dd}^1(\cA)$ is generated by its centre.
\end{proof}

\begin{theorem}
\label{thm: real_metric_calc}
    Let $(\Omega_{\dd}^1(\cA), \dagger, \Psi, \bra{\cdot}\ket{\cdot},\sigma^\mathrm{can})$ be a centred Hermitian differential calculus. Then the triple $(\LA{\cX}(\cA), \bra{\cdot}\ket{\cdot}_\cA, D_\ph)$, where $D_\ph := (D, \ph)$, is a right real metric calculus as in Definition \ref{defn:metric-calc}. Symmetrically the triple $(\RA{\cX}(\cA), \prescript{}{\A}{\bra{\cdot}\ket{\cdot}}, D_\ph)$ is a left real metric calculus. In both cases the inner products are strongly non-degenerate.
\end{theorem}

\begin{proof}
	We treat the right module (left-linear) case. Firstly, $\ph$ surjects onto $\LRA{\cX}(\cA)^\dag$ (see Proposition \ref{prop:real-der}), so using Lemma \ref{lem:der-gen} we find that $\ph(D)$ generates $\LA{\cX}(\cA)$ as a right $\A$-module. Secondly, we have to show that for all $\partial_1, \partial_2 \in \ph(D)$,
	\[
		\pairing{\ph(\partial_1)}{\ph(\partial_2)}_\cA = \pairing{\ph(\partial_1)}{\ph(\partial_2)}^*_\cA.
	\]
	But $\ph(\partial_i)$ is bilinear and real by Lemma \ref{def: def_map_phi} and Proposition \ref{prop:real-der}, so this is a direct consequence of Corollary \ref{cor: symmetry inner product}.
\end{proof}

\section{Relating  connections on vector fields and forms}
\label{sec:LC}

We now investigate how connections on forms relate to affine connections on vector fields, particularly in the context of centred Hermitian differential calculi. We start by recalling the respective formalisms of \cite{AWcurvature} and \cite{MRLC} for dealing with connections, and how 
metric compatibility and torsion freeness are expressed in each framework.

\subsection{Affine connections on modules of vector fields}

The following definitions and theorems are from \cite{AWcurvature}. We state them only for calculi over right $\A$-modules, but they all have analogs for calculi over left modules.

\begin{definition}
\label{def:aff-conn}
	Let $(\cX(\cA), \bra{\cdot}\ket{\cdot}_\cA, \fg_\ph)$ be a right metric calculus. A \emph{right affine connection} on $(\cX(\cA), \fg)$ is a $\C$-bilinear map $\nabla : \fg \times \cX(\cA) \rightarrow \cX(\cA)$ which for all $\lambda\in\C$, $\partial\in \fg$ and $X,Y\in \cX$ satisfies
	\[
	\begin{split}
		& \nabla_\partial (X + Y) = \nabla_\partial X + \nabla_\partial Y \\
		& \nabla_{\lambda \partial + \partial'} X = \lambda \nabla_\partial X + \nabla_\partial' X \\
		& \nabla_\partial (X a) = (\nabla_\partial X)a + X \partial a.
	\end{split}
	\]
	We call the data $(\cX(\cA), \bra{\cdot}\ket{\cdot}_\cA, \fg_\ph, \nabla)$ a \emph{right connection calculus}.
	
	If furthermore $(\cX(\cA), \bra{\cdot}\ket{\cdot}_\cA, \fg_\ph)$ is a right real metric calculus and $\pairing{\nabla_\partial X}{Y} = \pairing{\nabla_\partial X}{Y}^*$ for all $X, Y \in \ph(\fg)$ and $\partial \in \fg$, then we say that the data $(\cX(\cA), \bra{\cdot}\ket{\cdot}_\cA, \fg_\ph, \nabla)$ is a \emph{right real connection calculus}.
\end{definition}

\begin{definition}
\label{defn:torsion-metric-derivations}
	Let $(\cX(\cA), \bra{\cdot}\ket{\cdot}_\cA, \fg_\ph, \nabla)$ be a right real connection calculus. The calculus is called \emph{metric} if
	\[
		\partial \pairing{X}{Y}_\A = \pairing{\nabla_\partial X}{Y}_\A + \pairing{X}{\nabla_\partial Y}_\A
	\]
	for all $X, Y \in \cX(\cA)$, $\partial \in \fg$, and \emph{torsion-free} if
	\[
		\tau(\partial_1,\partial_2):=\nabla_{\partial_1} \ph(\partial_2) - \nabla_{\partial_2} \ph(\partial_1) - \ph([\partial_1, \partial_2]) 
	\]
	vanishes for all $\partial_1,\partial_2\in\fg$.
	If $(\cX(\cA), \bra{\cdot}\ket{\cdot}_\cA, \fg_\ph, \nabla)$ is both metric and torsion-free, we say it is a \emph{right pseudo-Riemannian calculus} over $\cX(\cA)$. If moreover $\pairing{\nabla_{\partial_1} \nabla_{\partial_2} X}{Y}_\cA$ is Hermitian for all $\partial_1, \partial_2 \in \fg$ and $X, Y \in \ph(\fg)$, then the pseudo-Riemannian calculus is called \emph{real}.
\end{definition}

%

\begin{example}
\label{eg:pseudo-vectorfields}
	Let $M$ be a smooth real manifold with a pseudo-Riemannian metric $g$. Denote by $\Gamma_\C(TM)$  the complexification of smooth sections of the tangent bundle. There is a one-to-one mapping $\ph : \DER(C^\infty(M)) \rightarrow \Gamma_\C(TM)$, where $C^\infty(M)$ is the algebra of smooth complex-valued functions on $M$. It is well-known that there is a unique metric compatible and torsion-free connection $\nabla$ on $\Gamma(TM)$, called the Levi-Civita connection, and it naturally extends complex-linearly to $\Gamma_\C(TM)$. 
	
	There are two approaches to extending the metric $g$ to $\Gamma_\C(TM)$: bi-linearly or sesqui-linearly. The latter corresponds to the approach of \cite{AWcurvature,AW2}, while the former corresponds to the bilinear ``quantum metric'' of Definition \ref{defn:metric-bilinear}, due to \cite{Gosw&alCentredBimod}. 
﻿
For the sesquilinear extension of the metric,
the calculus $(\Gamma_\C(TM), g, \DER(C^\infty(M))_\ph)$ is a real metric calculus and  the Levi-Civita connection makes $(\Gamma_\C(TM), g, \DER(C^\infty(M))_\ph, \nabla)$ into a (real) pseudo-Riemannian calculus.
\end{example}

\begin{example}
	The 2-torus $\T^2$ acts by isometries on $\T^2$ and $\SS^3$. The $\theta$-deformation construction (see \cite{CL}) yields one-parameter families of noncommutative algebras $C^\infty(\T^2_\theta)$ and $C^\infty(\Sf^3_\theta)$. In \cite{AWcurvature}, pseudo-Riemannian calculi are explicitly constructed for both of these families by analogy with the classical calculi. In subsection \ref{subsec:isospectral}  we return to these examples and generalise to any $\theta$-deformation coming from a free toral action on a Riemannian manifold, using differential forms and the method of \cite{MRLC}.
\end{example}

One would like  to prove that there exists a unique metric and torsion-free connection for every real metric calculus. In the framework of \cite{AWcurvature, AW2}, it has been shown that there exists at most one such connection. In \cite{AH25} a more general setup is considered, and sufficient conditions for existence are explored.

\begin{theorem}\cite[Theorem 3.4]{AWcurvature}
\label{thm: uniqueness vector fields}
	Let $(\cX(\cA), \bra{\cdot}\ket{\cdot}_\cA, \fg_\ph)$ be a right real metric calculus. Then there exists at most one affine connection $\nabla$ on $\cX(\cA)$ such that $(\cX(\cA), \bra{\cdot}\ket{\cdot}_\cA, \fg_\ph, \nabla)$ is a right pseudo-Riemannian calculus.
\end{theorem}

One of the main motivations for defining Levi-Civita connections on noncommutative algebras is to investigate curvature on these ``noncommutative spaces". With the formalism of \cite{AWcurvature}, curvature can be defined as usual in differential geometry.

\begin{definition}
	Let $(\cX(\cA), \bra{\cdot}\ket{\cdot}_\cA, \fg_\ph, \nabla)$ be a right pseudo-Riemannian calculus. The \emph{curvature operator} of $\nabla$ is defined on $\cX(\cA)$ by
	 \[
	 	R(\partial_1, \partial_2) = \nabla_{\partial_1} \nabla_{\partial_2} - \nabla_{\partial_2} \nabla_{\partial_1} - \nabla_{[\partial_1, \partial_2]}.
	 \]
\end{definition}


The definition of Ricci and scalar curvature is discussed in \cite{AWcurvature} with additional hypotheses. 
On the other hand, in \cite{MRWeitzenbock} Ricci and scalar curvature are generically defined in the context of Levi-Civita connections on noncommutative differential forms. We will show that when both frameworks apply, they are compatible, and so we can define Ricci and scalar curvature for pseudo-Riemannian calculi. 
We will discuss the full curvature tensors later, and refer to \cite{AWcurvature} and \cite{MRWeitzenbock} for discussions of Ricci and scalar curvatures.

\subsection{Connections on differential forms}
We now recall the standard definitions of (bimodule) connections on modules. In subsection \ref{subsec:exist} we will prove an existence and uniqueness theorem for Levi-Civita connections combining the frameworks of \cite{MRLC} and \cite{AWcurvature}, thus recovering in a new way a result of \cite{Gosw&alCentredBimod}. 
Starting from connections on differential forms and dualising then provides an existence proof for pseudo-Riemannian calculi coming from centred calculi.

Given a first order differential structure $(\Omega^1_\dee(\A),\dag)$, a right connection on a right $\A$-module $\X$ is a $\C$-linear map
\[
\nablar:\X\to \X\otimes_{\A}\Omega^{1}_{\dee},\quad\mbox{such that}\quad
\overrightarrow{\nabla}(x a)=\overrightarrow{\nabla}(x)a+x\ox \dee a.
\]
Similarly, a left connection on a left $\A$-module $\X$ is a $\C$-linear map
\[
\nablal:\X\to \Omega^{1}_{\dee}\otimes_{\A}\X,\quad\mbox{such that}\quad
\overrightarrow{\nabla}(a x)=a\overrightarrow{\nabla}(x)+\dee a\ox x.
\]
Connections always exist on finite projective modules \cite{CQ,Landi}. In the presence of an Hermitian inner product, we have the well-known definition of Hermitian connections which we recall below. We first need to introduce a notation for pairings between tensors of different orders:

\begin{notation}
    For $\omega, \eta, \rho \in \Omega^1_\dee(\cA)$, we define
    \begin{align*}
        & \pairing{\omega}{\eta \ox \rho}_{\Omega^1_\dee(\cA)} := \pairing{\omega}{\eta}_\cA \rho \\
        & \pairing{\eta \ox \rho}{\omega}_{\Omega^1_\dee(\cA)} := \rho^* \pairing{\eta}{\omega}_\cA.
    \end{align*}
    We then extend this by linearity to $\pairing{\omega}{\alpha}_{\Omega^1_\dee(\cA)}$ and $\pairing{\alpha}{\omega}_{\Omega^1_\dee(\cA)}$ for $\omega \in \Omega^1_\dee(\cA)$ and $\alpha \in T^2_\dee(\cA)$. We also define left pairings by ${}_{\Omega^1_\dee(\cA)}\pairing{\.}{\.} := \pairing{\.^\dag}{\.^\dag}_{\Omega^1_\dee(\cA)}$.
\end{notation}

\begin{definition}Given a connection $\overrightarrow{\nabla}$ on a right inner product $\A$-module $\X$ we say that $\overrightarrow{\nabla}$ is Hermitian if
for all $x,y\in\X$ we have
\[
-\pairing{\overrightarrow{\nabla}x}{y}_{\Omega^1_\dee(\cA)} + \pairing{x}{\overrightarrow{\nabla}y}_{\Omega^1_\dee(\cA)} = \dee\left(\pairing{x}{y}_\A \right).
\]
For left connections we instead require
\[
{}_{\Omega^1_\dee(\cA)}\pairing{\overleftarrow{\nabla}x}{y}-{}_{\Omega^1_\dee(\cA)} \pairing{x}{\overleftarrow{\nabla}y} = \dee\left({}_\A \pairing{x}{y}\right).
\]
\end{definition}
The differential \eqref{eq: second-order-diff} allows us to  ask whether a connection on $\Omega^1_\dee(\A)$ is torsion-free, leading to the well-known definitions from the algebraic literature \cite{Landi, BMbook}.
\begin{definition}
\label{defn:torsion on forms}
Let $(\Omega^{1}_{\dee},\dag,\Psi)$ be an Hermitian second-order differential structure as in Definition \ref{ass:fgp-metric-junk} and $\nablar:\Omega^{1}_{\dee}\to T^{2}_{\dee}$ a right connection and $\nablal:\Omega^{1}_{\dee}\to T^{2}_{\dee}$ a left connection. The \emph{torsion} of $\nablar$ and $\nablal$, respectively, are the maps
\begin{align*}
(1-\Psi)\circ\nablar+\d\in\overrightarrow{\textnormal{Hom}}_{\A}(\Omega^{1}_{\dee},T^{2}_{\dee})\,\,\\
(1-\Psi)\circ\nablal-\d\in\overleftarrow{\textnormal{Hom}}_{\A}(\Omega^{1}_{\dee},T^{2}_{\dee}).
\end{align*}
\end{definition}
Since $\Omega^{1}_{\dee}$ and $T^{2}_{\dee}$ are $\dag$-bimodules, for each right connection $\overrightarrow{\nabla}:\Omega^{1}_{\dee}\to T^{2}_{\dee}$ there is a conjugate left connection $\overleftarrow{\nabla}$ given by
$\overleftarrow{\nabla}=-\dag\circ\overrightarrow{\nabla}\circ\dag$ which is Hermitian if and only if $\overrightarrow{\nabla}$ is Hermitian.
\begin{definition}
\label{ass:ess}
Suppose that $\sigma: T^{2}_\dee(\A)\to T^2_\dee(\A)$ is a braiding, that is, an invertible bimodule map such that $\dag\circ\sigma=\sigma^{-1}\circ\dag$. If 
the conjugate connections 
$\nablar,\nablal$ satisfy
\[
\sigma\circ\overrightarrow{\nabla}=\overleftarrow{\nabla},
\]
then we say that $(\nablar,\sigma)$ is a $\sigma$-$\dag$-bimodule connection.
\end{definition}
\begin{remark} The notion of a $\sigma$-$\dag$-bimodule connection can be defined for general $\dag$-bimodules $\X$ equipped with a braiding $\sigma:\X\otimes_{\A}\Omega^{1}_{\dee}(\A)\to \Omega^{1}_{\dee}(\A)\otimes_{\A}\X$, see \cite[Definition 5.8]{MRLC}. For the present paper it suffices to confine ourselves to the case $\X=\Omega^{1}_{\dee}$.
\end{remark}

\subsection{Dualising noncommutative connections}
\label{subsec:dualising}

In differential geometry one can relate the evaluation of a given connection, or covariant derivative, on differential one-forms and on vector fields. This can be expressed locally using the Christoffel symbols or globally in terms of duality. It is this latter form that we extend, first on general inner product $\dag$-bimodules of forms and then specifically for centred Hermitian calculi.

Throughout this section, we fix $\cA$ a unital $*$-algebra and $(\Omega_{\dd}^1(\cA), \bra{\cdot}\ket{\cdot}_\cA, \dagger)$ an Hermitian first-order differential structure over $\cA$.

\begin{proposition}
\label{prop:left-right-left-right}
	 Suppose $\RA{\nabla}$ (resp. $\LA{\nabla}$) is a right (resp. left) connection on $(\Omega_{\dd}^1(\cA), \bra{\cdot}\ket{\cdot}_\cA, \dagger)$. For $X \in \RA{\cX}(\cA)$
	 and $\partial \in D$, $\omega \in \Omega_{\dd}^1(\cA)$, the formula
	 \[
	 \Pairing{\RA{\nabla}_\partial X}{\omega} := \partial \Pairing{X}{\omega} - \product{X}{\RA{\nabla}(\omega)}{\ph(\partial)}
	 \]
	 defines a left affine connection on the left $\A$-module $\RA{\cX}(\cA)$. Similarly, for $X \in \LA{\cX}(\cA)$
	 and $\partial \in D$, $\omega \in \Omega_{\dd}^1(\cA)$, the formula
	 \[
	 \Pairing{\omega}{\LA{\nabla}_\partial X} := \partial \Pairing{\omega}{X} - \product{\ph(\partial)}{\LA{\nabla}(\omega)}{X}
	 \]
	 defines a right affine connection on the right $\A$-module $\LA{\cX}(\cA)$.
\end{proposition}

\begin{proof}
	We only show the computation for left-linear vector fields, the right being analogous. First we show that $\LA{\nabla}_\partial X$ is a left-linear vector field. For $X \in \LA{\cX}(\A)$ and $a \in \cA$,
	\[
	\begin{split}
		\Pairing{a \omega}{\LA{\nabla}_\partial X}
		&= \partial \Pairing{a \omega}{X} - \product{\ph(\partial)}{\LA{\nabla}(a \omega)}{X} \\
		&= (\partial a) \Pairing{\omega}{X} + a \partial \Pairing{\omega}{X} - \product{\ph(\partial)}{a \LA{\nabla}(\omega)}{X} - \product{\ph(\partial)}{\dd a \otimes \omega}{X}\\
		&= (\partial a) \Pairing{\omega}{X} + a \partial \Pairing{\omega}{X} - \product{\ph(\partial)}{a \LA{\nabla}(\omega)}{X} - \Pairing{\ph(\partial)}{\dee a} \Pairing{\omega}{X}\\
		&= (\partial a) \Pairing{\omega}{X} + a \left( \partial \Pairing{\omega}{X} - \product{\ph(\partial)}{\LA{\nabla}(\omega)}{X} \right) - (\partial a) \Pairing{\omega}{X}\\
		&=a \left( \partial \Pairing{\omega}{X} - \product{\ph(\partial)}{\LA{\nabla}(\omega)}{X} \right) = a\Pairing{ \omega}{\LA{\nabla}_\partial X}.
	\end{split}
	\]
    The fourth equality comes from the identity $\product{\ph(\partial)}{a \eta \otimes_\cA \rho}{X} = a \product{\ph(\partial)}{\eta \otimes_\cA \rho}{X}$ which holds because $\ph(\partial)$ is bilinear. We deduce that $\LA{\nabla}_\partial X$ is a left-linear vector field. It is clearly $\C$-linear in $\partial$. Moreover, for any $b \in \cA$,
	\[
	\begin{split}
	    \Pairing{\omega}{\LA{\nabla}_\partial (Xb)} &= \partial \Pairing{\omega}{Xb}- \product{\ph(\partial)}{\LA{\nabla}(\omega)}{Xb} \\
	    &= \Pairing{\omega}{X} (\partial b) + \big( \partial \Pairing{\omega}{X} - \product{\ph(\partial)}{\LA{\nabla}(\omega)}{X} \big) b \\
	    &= \Pairing{\omega}{X (\partial b) + (\LA{\nabla}_\partial X) b},
	\end{split}
	\]
	which finishes the proof.
\end{proof}

\begin{proposition}\label{prop: dagger and affine connections}
    Let $\RA{\nabla}$ and $\LA{\nabla} := - \dagger \circ \RA{\nabla} \circ \dagger$ be a conjugate pair of connections on $(\Omega_{\dd}^1(\cA), \bra{\cdot}\ket{\cdot}_\cA, \dagger)$. Then the dagger intertwines the corresponding affine connections on vector fields, so that
    \[
        (\RA{\nabla}_\partial X)^\dagger = \LA{\nabla}_\partial (X^\dagger)
    \]
    for all $\partial \in D$ and $X \in \RA{\cX}(\cA)$.
\end{proposition}

\begin{proof}
    For $\partial \in D$ and $X \in \RA{\cX}(\cA)$, we compute for any $\omega \in \Omega_{\dd}^1(\cA)$,
    \[
        \Pairing{\omega}{( \RA{\nabla}_\partial X )^\dagger} = \Pairing{\RA{\nabla}_\partial X}{\omega^\dagger}^{*} = \left[ \partial \Pairing{X}{\omega^\dagger} \right]^* - \product{X}{\RA{\nabla}(\omega^\dagger)}{\ph(\partial)}^*.
    \]
    If we use Sweedler notation to write $\RA{\nabla}(\omega^\dagger)=\eta_{(1)}\ox\eta_{(2)}$, then since $\RA{\nabla}(\omega^\dagger) = -\LA{\nabla}(\omega)^\dagger$ we find
    \[
    -\nablal(\omega) = \nablar(\omega^\dag)^\dag=(\eta_{(1)}\ox\eta_{(2)})^\dag=\eta_{(2)}^\dag\ox\eta_{(1)}^\dag.
    \]
    With this computation in hand
    \begin{align*}
    \product{X}{\RA{\nabla}(\omega^\dagger)}{\ph(\partial)}^*
    &=\big(\Pairing{X}{\eta_{(1)}}\Pairing{\eta_{(2)}}{\varphi(\partial)}\big)^*\\
    &=\Pairing{\varphi(\partial)^\dag}{\eta_{(2)}^\dag}\Pairing{\eta_{(1)}^\dag}{X^\dag}\\
    &=\product{\ph(\partial)^\dag}{\eta_{(2)}^\dag\ox\eta_{(1)}^\dag}{X^\dag}\\
    &=\product{\ph(\partial)^\dag}{-\nablal(\omega)}{X^\dag}.
    \end{align*} 
    Using $\partial(a)^* = \partial(a^*)$, we get
    \[
        \Pairing{\omega}{( \RA{\nabla}_\partial X )^\dagger} = \partial \Pairing{\omega}{X^\dagger} + \product{\ph(\partial)^\dag}{\LA{\nabla}(\omega)}{X^\dag} = \Pairing{\omega}{\LA{\nabla}_\partial (X^\dagger)},
    \]
    the last equality following from $\ph(\partial)^\dagger = -\ph(\partial)$.
\end{proof}

\begin{proposition}
\label{prop: connections on bilinear vector fields}
    Let  $(\RA{\nabla}, \sigma)$ be a $\dagger$-bimodule connection on $(\Omega_{\dd}^1(\cA), \bra{\cdot}\ket{\cdot}_\cA, \dagger)$, and assume that $\Omega_{\dd}^1(\cA)$ is centred so that $\sigma=\sigma^{\mathrm{can}}$. Then for all bilinear vector fields $X \in \LRA{\cX}(\cA)$ and derivations $\partial \in D$,
    \[
        \RA{\nabla}_\partial X = \LA{\nabla}_\partial X,
    \]
    where $\nablar_\partial$ and $\nablal_\partial$ are the affine connections induced by $\nablar$ and its conjugate $\nablal = -\dag \circ \nablar \circ \dag$.
\end{proposition}

\begin{proof}
    For all one-forms $\omega$ we have $\partial \Pairing{X}{\omega}=\partial \Pairing{\omega}{X}$ since $X$ is bilinear, so the first terms in the definitions of $\nablar_\partial$ and $\nablal_\partial$ are the same. We are then left with showing that the second terms are identical, namely
    \[
        \product{X}{\RA{\nabla}(\omega)}{\ph(\partial)} = \product{\ph(\partial)}{\LA{\nabla}(\omega)}{X}.
    \]
    By definition of a $\dagger$-bimodule connection, $\RA{\nabla} = \sigma^{-1} \circ \LA{\nabla}$. Now, for any two bilinear vector fields $U, V$ and central one-forms $\eta, \rho$,
    \[
        \product{U}{\sigma^{-1}(\eta \otimes \rho)}{V} = \product{U}{\rho \otimes \eta}{V} = \Pairing{U}{\rho} \Pairing{\eta}{V}.
    \]
    Using bilinearity of $U$ and centrality of $\rho$, Lemma \ref{lem:bilinear-central} shows that $\Pairing{U}{\rho} \in \cZ(\cA)$, and so
    \[
        \product{U}{\sigma^{-1}(\eta \otimes \rho)}{V} = \Pairing{V}{\eta} \Pairing{\rho}{U} = \product{V}{\eta \otimes \rho}{U}.
    \]
    Using that $\Omega_{\dd}^1(\cA)$ is centred we deduce that this equality is true for any two-tensor $\alpha \in T^2_{\dd}(\cA)$. Specifying to $\alpha = \RA{\nabla}(\omega)$ yields the result.
\end{proof}

As for the musical isomorphisms, when working over a centred bimodule of forms with $\sigma = \sigma^\text{can}$, we will omit arrows when evaluating $\RA{\nabla}_\partial$ or $\LA{\nabla}_\partial$ on $\LRA{\cX}(\cA)$.

\begin{lemma}
\label{lem:bi-real}
    Let  $(\RA{\nabla}, \sigma)$ be a $\dagger$-bimodule connection on $(\Omega_{\dd}^1(\cA), \bra{\cdot}\ket{\cdot}_\cA, \dagger)$, and assume that $\Omega_{\dd}^1(\cA)$ is centred so that $\sigma=\sigma^{\mathrm{can}}$. If $X$ is a real bilinear vector field and $\partial \in D$ an Hermitian derivation, then $\nabla_\partial X$ is bilinear and real.
\end{lemma}

\begin{proof}
    If $X$ is a bilinear vector field, then $\nablar_\partial X$ and $\nablal_\partial X$ are by construction respectively a right- and a left-linear vector field. Since they coincide according to Proposition \ref{prop: connections on bilinear vector fields}, then the common value $\nabla_\partial X$ is bilinear. If furthermore $X$ is real, then using Proposition \ref{prop: dagger and affine connections} we see that $(\nablar_\partial X)^\dagger = \nablal_\partial(X^\dag)$, i.e. $(\nabla_\partial X)^\dag = \nabla_\partial(X^\dag) = -\nabla_\partial X$ since $X^\dag = -X$.
\end{proof}

\begin{theorem}\label{thm: real connection calculus}
    Let $(\Omega_{\dd}^1(\cA), \dagger, \Psi, \bra{\cdot}\ket{\cdot},\sigma^{\mathrm{can}})$ be a centred Hermitian differential calculus and $(\RA{\nabla}, \sigma^{\mathrm{can}})$ be a $\dagger$-bimodule connection on $\Omega^{1}_{\mathrm{d}}(\A)$. For all $X, Y \in \ph(D)$ and $\partial \in D$,
    \[
        \pairing{\nabla_\partial X}{Y}_\cA^* = \pairing{\nabla_\partial X}{Y}_\cA.
    \]
    Therefore $(\LA{\cX}(\cA), \bra{\cdot}\ket{\cdot}_\cA, D_\ph, \LA{\nabla})$ and $(\RA{\cX}(\cA), \prescript{}{\A}{\bra{\cdot}\ket{\cdot}}, D_\ph, \RA{\nabla})$ are respectively a right and a left real connection calculus, and the inner products are strongly non-degenerate.
\end{theorem}

\begin{proof}
    We give the proof for the right calculus. Hermitianness of $\pairing{\nabla_\partial X}{Y}_\cA$ follows from Lemma \ref{lem:bi-real}, using the fact that $\ph(D) \subset \LRA{\cX}(\cA)$ and that the inner product is symmetric on real bilinear vector fields (see Corollary \ref{cor: symmetry inner product}). Then, $(\LA{\cX}(\cA), \bra{\cdot}\ket{\cdot}_\cA, D_\ph)$ is a right real metric calculus by Theorem \ref{thm: real_metric_calc}, so by the Hermitianness of $\pairing{\nabla_\partial X}{Y}_\cA$, adding $\LA{\nabla}$ yields a right real connection calculus.
\end{proof}

\subsection{Dual Levi-Civita connections}

In this section we investigate metric compatibility and torsion of connections over centred Hermitian differential calculi. We show that the definitions of Levi-Civita connections in \cite{MRLC} and \cite{AWcurvature} coincide on these calculi, and that torsion maps are related by classical formulae.

In this section we fix $(\Omega_{\dd}^1(\cA), \dagger, \Psi, \bra{\cdot}\ket{\cdot},\sigma^\mathrm{can})$ a centred Hermitian differential calculus over a unital $*$-algebra $\A$ and $(\nablar,\nablal)$ a conjugate pair of connections over $\Omega_{\dd}^1(\cA)$.

\begin{lemma}
\label{lem:pair-prods}
	Let $U$ be a left-linear vector field on $\Omega_{\dd}^1(\cA)$, $\omega \in \Omega_{\dd}^1(\cA)$ and $\alpha \in T^2_{\dd}(\cA)$. Then
	\[
		\Pairing{\pairing{\omega}{\alpha}_{\Omega_{\dd}^1(\cA)}}{U} =  \product{(\omega^{\LA{\sharp}})^\dag}{\alpha}{U}.
	 \]
\end{lemma}

\begin{proof}
    We prove the statement for a simple tensor $\alpha = \eta \otimes \rho$ and extend by $\C$-linearity. Recall that by definition, $\pairing{\omega}{\eta \ox \rho}_{\Omega_\dee^1(\cA)} = \pairing{\omega}{\eta}_\cA \rho$. So, by left-$\cA$-linearity of $U$,
    \[
    	\Pairing{\pairing{\omega}{\eta \ox_\cA \rho}_{\Omega_{\dd}^1(\cA)} }{U}= \pairing{\omega}{\eta}_\cA \Pairing{\rho}{U} = \Pairing{(\omega^{\sharpl})^\dag}{\eta} \Pairing{\rho}{U} = \product{(\omega^{\sharpl})^\dag}{\eta \ox \rho}{U}.\qedhere
    \]
\end{proof}

\begin{proposition}\label{prop: hermitian connections}
	Suppose the conjugate pair $(\nablar,\nablal)$ is Hermitian. Then for any Hermitian derivation $\partial \in D$, the affine connections $\LA{\nabla}_\partial$ and $\RA{\nabla}_\partial$ are metric, i.e.
	\[
	\begin{split}
	    &\partial \pairing{X}{Y}_\cA = \pairing{\LA{\nabla}_\partial X}{Y}_\cA + \pairing{X}{\LA{\nabla}_\partial Y}_\cA \\
	    &\partial \prescript{}{\cA}{\pairing{U}{V}} = \prescript{}{\cA}{\pairing{\RA{\nabla}_\partial U}{V}} + \prescript{}{\cA}{\pairing{U}{\RA{\nabla}_\partial Y}},
	\end{split}
	\]
	for all $X, Y \in \LA{\cX}(\cA)$ and $U, V \in \RA{\cX}(\cA)$.
\end{proposition}

\begin{proof}
    We show the computation for left-linear vector fields. Let $X, Y \in \LA{\cX}(\cA)$. Then
    \[
    \begin{split}
        \pairing{\LA{\nabla}_\partial X}{Y}_\cA + \pairing{X}{\LA{\nabla}_\partial Y}_\cA &= \Pairing{( \LA{\nabla}_\partial X )^\dagger}{Y^{\LA{\flat}}} + \Pairing{( X^{\LA{\flat}} )^\dagger}{\LA{\nabla}_\partial Y} \\
        &= \Pairing{(Y^{\LA{\flat}})^\dagger}{\LA{\nabla}_\partial X}^{*} + \Pairing{(X^{\LA{\flat}})^\dagger}{\LA{\nabla}_\partial Y} \\
        &= \partial \Pairing{(Y^{\LA{\flat}})^\dagger}{X}^* - \product{\ph(\partial)}{\LA{\nabla}\big( (Y^{\LA{\flat}})^\dagger \big)}{X}^* \\
        &\quad + \partial \Pairing{(X^{\LA{\flat}})^\dagger}{Y} - \product{\ph(\partial)}{\LA{\nabla}\big( (X^{\LA{\flat}})^\dagger \big)}{Y} \\
        &= 2 \partial \pairing{X}{Y}_\cA - \product{X^\dag}{\RA{\nabla} ( Y^{\LA{\flat}} )}{\ph(\partial)} - \product{Y^\dagger}{\RA{\nabla} ( X^{\LA{\flat}} )}{\ph(\partial)}^*.
    \end{split}
    \]
    The first and last equalities are a consequence of the computation
    \[
        \Pairing{U^\dagger}{V^{\LA{\flat}}}
        = \Pairing{\dagger \circ \LA{\sharp} \circ \LA{\flat} (U)}{V^{\LA{\flat}}}
        = \Pairing{\RA{\sharp} \circ \dagger \circ \LA{\flat} (U)}{V^{\LA{\flat}}}
        = \pairing{U^{\LA{\flat}}}{V^{\LA{\flat}}}_\cA = \pairing{U}{V}_\cA
    \]
    for $U, V \in \LA{\cX}(\A)$. Using Lemma \ref{lem:pair-prods}, we see that
    \[
        \product{X^\dagger}{\RA{\nabla} ( Y^{\LA{\flat}} )}{\ph(\partial)} = \Pairing{\pairing{X^{\LA{\flat}}}{\RA{\nabla} ( Y^{\LA{\flat}} )}_{\Omega_{\dd}^1(\cA)}}{\ph(\partial)}= \Pairing{\ph(\partial)}{\pairing{X^{\LA{\flat}}}{\RA{\nabla} ( Y^{\LA{\flat}} )}_{\Omega_{\dd}^1(\cA)}}
    \]
        since $\ph(\partial)$ is bilinear. Hence also, using $\ph(\partial)^\dag = -\ph(\partial)$,
    \[
        \product{Y^\dagger}{\RA{\nabla} ( X^{\LA{\flat}} )}{\ph(\partial)}^* = -\Pairing{\ph(\partial)}{\pairing{X^{\LA{\flat}}}{\RA{\nabla} ( Y^{\LA{\flat}} )}^\dagger_{\Omega_{\dd}^1(\cA)}}.
    \]
    Since $\pairing{\omega}{\eta \otimes \rho}_{\Omega_{\dd}^1(\cA)}^\dagger = \rho^\dagger \pairing{\omega}{\eta}^*_\cA = \pairing{\eta \otimes \rho}{\omega}_{\Omega_{\dd}^1(\cA)}$, we deduce
    \[
        \pairing{\LA{\nabla}_\partial X}{Y}_\cA + \pairing{X}{\LA{\nabla}_\partial Y}_\cA = 2 \partial \pairing{X}{Y}_\cA - \Pairing{\ph(\partial)}{\big[ \pairing{X^{\LA{\flat}}}{\RA{\nabla} ( Y^{\LA{\flat}} )}_{\Omega_{\dd}^1(\cA)} - \pairing{\RA{\nabla} ( X^{\LA{\flat}} )}{Y^{\LA{\flat}}}_{\Omega_{\dd}^1(\cA)} \big]}.
    \]
    Since $\RA{\nabla}$ is Hermitian,
    \[
        \pairing{X^{\LA{\flat}}}{\RA{\nabla} ( Y^{\LA{\flat}} )}_{\Omega_{\dd}^1(\cA)} - \pairing{\RA{\nabla} ( X^{\LA{\flat}} )}{Y^{\LA{\flat}}}_{\Omega_{\dd}^1(\cA)} = \dd \pairing{X^{\LA{\flat}}}{Y^{\LA{\flat}}}_\cA = \dd \pairing{X}{Y}_\cA,
    \]
    so applying $\ph(\partial)$ yields
    \[
        \pairing{\LA{\nabla}_\partial X}{Y}_\cA + \pairing{X}{\LA{\nabla}_\partial Y}_\cA = 2 \partial \pairing{X}{Y}_\cA - \Pairing{\ph(\partial)}{\dd \pairing{X}{Y}_\cA} = \partial \pairing{X}{Y}_\cA.\qedhere
    \]
\end{proof}

\begin{proposition}\label{prop: torsion coincide}
Let $(\nablar,\nablal)$ be a conjugate pair of connections on $\Omega^1_\dee(\A)$.
	The torsion map on derivations from Definition \ref{defn:torsion-metric-derivations} is given by
    \[
    \begin{split}
        \Pairing{\tau(\partial_1, \partial_2)}{\omega} &= 2 \product{\ph(\partial_2)}{\big( (1 - \Psi) \circ \LA{\nabla} - \dd_\Psi \big) (\omega)}{\ph(\partial_1)} \\
        &= 2 \product{\ph(\partial_1)}{\big( (1 - \Psi) \circ \RA{\nabla} + \dd_\Psi \big) (\omega)}{\ph(\partial_2)},
    \end{split}
    \]
    for all $\partial_1, \partial_2 \in D$ and $\omega \in \Omega_\dee^1(\cA)$.
\end{proposition}

\begin{proof}
   Recall that by definition, if $\partial_1, \partial_2 \in D$ are two Hermitian derivations then
    \[
    	\tau(\partial_1, \partial_2) = \nabla_{\partial_1} \big( \ph(\partial_2) \big) - \nabla_{\partial_2} \big( \ph(\partial_2) \big) - \ph \big( [\partial_1, \partial_2] \big).
    \]
    In fact, one might define a right and left torsion map, using the right and the left affine connections, but these coincide on the image of $\ph$.
    Hence, for all $\omega \in \Omega_\dee^1(\cA)$, using the definition of the left affine connection we have
    \begin{alignat*}{2}
        \Pairing{\tau(\partial_1, \partial_2)}{\omega} &=\ && \partial_1 \Pairing{\ph(\partial_2)}{\omega} - \partial_2 \Pairing{\ph(\partial_1)}{\omega} - \Pairing{ \ph \big( [\partial_1, \partial_2] \big) }{\omega} \\
        & && - \product{\ph(\partial_1)}{\LA{\nabla}(\omega)}{\ph(\partial_2)} + \product{\ph(\partial_2)}{\LA{\nabla}(\omega)}{\ph(\partial_1)}.
    \end{alignat*}
    Since $\tau(\partial_1, \partial_2)$ is clearly $\cA$-bilinear, it is determined by its action on exact forms. Thus, if $b \in \cA$,
    \begin{alignat*}{2}
        \pairing{\tau(\partial_1, \partial_2)}{\dd b} &=\ && \partial_1(\partial_2 b) - \partial_2(\partial_1 b) - [\partial_1, \partial_2]b \\
        & && - \product{\ph(\partial_1)}{\LA{\nabla}(\dd b)}{\ph(\partial_2)} + \product{\ph(\partial_2)}{\LA{\nabla}(\dd b)}{\ph(\partial_1)} \\
        &= && \product{\ph(\partial_2)}{\LA{\nabla}(\dd b)}{\ph(\partial_1)} - \product{\ph(\partial_1)}{\LA{\nabla}(\dd b)}{\ph(\partial_2)}.
    \end{alignat*}
    Now consider a generic one-form $\omega = \sum_i a_i \dd b_i$. By definition  we have $a_i \LA{\nabla}(\dd b_i) = \LA{\nabla}(a_i \dd b_i) - \dd a_i \otimes \dd b_i$. So, using left-linearity of $\tau(\partial_1, \partial_2)$ we get
    \[
        \Pairing{\tau(\partial_1, \partial_2)}{\omega} = \product{\ph(\partial_2)}{\LA{\nabla}(\omega) - \dd a_i \otimes \dd b_i}{\ph(\partial_1)} - \product{\ph(\partial_1)}{\LA{\nabla}(\omega) - \dd a_i \otimes \dd b_i}{\ph(\partial_2)}.
    \]
    Lemma \ref{lemma: antisym_vector_fields} then yields the first formula 
    \[
        \Pairing{\tau(\partial_1, \partial_2)}{\omega} = 2 \product{\ph(\partial_2)}{(1 - \Psi) \circ \LA{\nabla} (\omega) - \dd_\Psi \omega}{\ph(\partial_1)}.
    \]
    For the second formula, we proceed in a similar manner: using the definition of the right affine connection,
    \begin{alignat*}{2}
        \Pairing{\tau(\partial_1, \partial_2)}{\omega} &=\ && \partial_1 \Pairing{\ph(\partial_2)}{\omega} - \partial_2 \Pairing{\ph(\partial_1)}{\omega} - \Pairing{ \ph \big( [\partial_1, \partial_2] \big) }{\omega} \\
        & && - \product{\ph(\partial_2)}{\RA{\nabla}(\omega)}{\ph(\partial_1)} + \product{\ph(\partial_1)}{\RA{\nabla}(\omega)}{\ph(\partial_2)}.
    \end{alignat*}
    On an exact form $\omega = \dee b$, the first three terms compensate:
    \[
       \Pairing{\tau(\partial_1, \partial_2)}{\dee b} = \product{\ph(\partial_1)}{\RA{\nabla}(\dee b)}{\ph(\partial_2)} - \product{\ph(\partial_2)}{\RA{\nabla}(\dee b)}{\ph(\partial_1)}.
    \]
    For a generic one-form we write $\omega = \sum a_i \dee b_i = \sum \dee(a_i b_i) - (\dee a_i) b_i$. Then $\nablar\big[ \dee(a_i b_i) \big] - \nablar(\dee a_i) b_i = \nablar\big[ \dee(a_i b_i) \big] - \nablar\big[ (\dee a_i) b_i \big] + \dee a_i \ox \dee b_i = \nablar(\omega) + \dee a_i \ox \dee b_i$. Hence, using bilinearity of $\tau(\partial_1, \partial_2)$,
    \[
        \Pairing{\tau(\partial_1, \partial_2)}{\omega} = \product{\ph(\partial_1)}{\RA{\nabla}(\omega) + \dee a_i \ox \dee b_i}{\ph(\partial_2)} - \product{\ph(\partial_2)}{\RA{\nabla}(\omega) + \dee a_i \ox \dee b_i}{\ph(\partial_1)}
    \]
    and we conclude as above using Lemma \ref{lemma: antisym_vector_fields}.
\end{proof}

\begin{theorem}\label{thm: pseudo-Riemannian_calculus}
	If $(\RA{\nabla}, \sigma)$ is an Hermitian and torsion-free $\sigma$-$\dag$-bimodule connection on $\Omega^{1}_{d}(\A)$ then the data $(\LA{\cX}(\cA), \bra{\cdot}\ket{\cdot}_\cA, D_\ph, \LA{\nabla})$ defines a right real pseudo-Riemannian calculus and $(\RA{\cX}(\cA), \prescript{}{\A}{\bra{\cdot}\ket{\cdot}}, D_\ph, \RA{\nabla})$ defines a left real pseudo-Riemannian calculus, and the inner products are strongly non-degenerate.
\end{theorem}

\begin{proof}
    We show the proof for the right calculus. For $(\RA{\nabla}, \sigma)$ Hermitian and torsion-free we have $(1 - \Psi)\RA{\nabla} = -\dd_\Psi$ (see \ref{defn:torsion on forms}). It follows that $\nablal_\partial$ is torsion-free by Proposition \ref{prop: torsion coincide}, and it is metric by Proposition \ref{prop: hermitian connections}. By Theorem \ref{thm: real connection calculus}, $(\LA{\cX}(\cA), \bra{\cdot}\ket{\cdot}_\cA, D_\ph, \LA{\nabla})$ is a real connection calculus. Therefore it satisfies the conditions of a pseudo-Riemannian calculus.

    To prove we have a \emph{real} pseudo-Riemannian calculus, we only have left to show that $\pairing{\nabla_{\partial_1} \nabla_{\partial_2} X}{Y}_\cA^* = \pairing{\nabla_{\partial_1} \nabla_{\partial_2} X}{Y}_\cA$ for any two Hermitian connections $\partial_1, \partial_2$ and $X, Y \in \ph(D)$. The proof of this is identical to that of Theorem \ref{thm: real connection calculus}: $X, Y$ are real bilinear vector fields and $\nabla_\partial$ preserves $\LRA{\cX}(\A)^\dag$, so $\nabla_{\partial_1} \nabla_{\partial_2} X$ is also real bilinear. The property then follows from symmetry of the inner product on $\LRA{\cX}(\A)^\dag$.
\end{proof}
\begin{corollary}
\label{cor:unique}
    If $(\RA{\nabla}, \sigma)$ and $(\RA{\nabla'}, \sigma)$ are two Hermitian and torsion-free $\sigma$-$\dagger$-bimodule connections on $\Omega_{\dd}^1(\cA)$, then $\RA{\nabla} = \RA{\nabla'}$.
\end{corollary}

\begin{proof}
    This is a direct consequence of Theorem \ref{thm: pseudo-Riemannian_calculus} and \cite[Theorem 3.4]{AWcurvature}, which ensures uniqueness of pseudo-Riemannian calculi on vector fields. Indeed, the formulae expressing affine connections on vector fields in terms of connections on differential forms can be inverted because the image of $D$ by $\ph$ generates $\LRA{\cX}(\cA)$ $\C$-linearly, which in turn generates both $\RA{\cX}(\cA)$ and $\LA{\cX}(\cA)$ as modules (and these are isomorphic to $\Omega_{\dd}^1(\cA)$).
\end{proof}

\begin{remark}
\label{rmk:MR-unique}
In the next section we present an existence proof for Hermitian torsion-free connections on one-forms based on applying \cite{MRLC} in the centred and strongly-nondegenerate setting. The uniqueness criteria given in \cite[Theorem 5.14]{MRLC} is easily seen to hold using that $\sigma$ is the flip-map on central forms, an argument analogous to \cite[Lemma 6.11]{MRLC} and the invariance of the quantum metric $g$ under $\sigma$. This would give an alternative and independent proof of uniqueness.
\end{remark}

\subsection{Existence of the Levi-Civita connection}
\label{subsec:exist}
We now provide a general existence and uniqueness result for centred Hermitian differential calculi, inspired by the framework of \cite{MRLC}, but adapted to the algebraic setting of \cite{Gosw&alCentredBimod, BMbook}. This recovers and extends \cite[Theorem 4.1]{Gosw&alCentredBimod} allowing any centred bimodule of one-forms with strongly non-degenerate inner product, though the approach of the proofs differs. In \cite[Theorem 4.1]{Gosw&alCentredBimod}, a torsion-free connection is produced, and then a correction added to make it Hermitian. We produce an Hermitian connection using our replacements for frames, and then add a correction to obtain a torsion-free connection.


With duality in hand, Theorem \ref{thm: existence and uniqueness on centred bimods} will provide an existence statement for Levi-Civita connections on pseudo-Riemannian calculi coming from centred Hermitian calculi. Uniqueness will then follow from Corollary \ref{cor:unique} or directly from the methods of \cite{MRLC} as in Remark \ref{rmk:MR-unique}.

The key results that allow us to replace the analytic considerations of \cite{MRLC} with strong non-degeneracy are contained in Appendix \ref{sec:fgpipdb}. There we show that finitely generated modules with a strongly non-degenerate inner product, in the sense of Definition \ref{defn:first-order}, have (pairs of) generating sets which serve the same purpose as frames in $C^*$-modules.

\begin{proposition} 
\label{prop: Grassmann}
Let $(\Omega^1_{\dee}(\A),\dag,\pairing{\cdot}{\cdot},\sigma^{\mathrm{can}})$ be a centred Hermitian differential calculus and $\{(\omega_{i},\eta_{i})\}$, a pair of central generating sequences such that $\{(\omega_{i},\eta_{i})\}=\{(\omega_{i}^{\dag},\eta_{i}^{\dag})\}$ as sets and for all $\omega\in\Omega^{1}_{\dee}$
\[
\omega=\sum_{i}\omega_{i}\pairing{\eta_{i}}{\omega}=\sum_{i}\eta_{i}\pairing{\omega_{i}}{\omega}.
\] 
Such $\dag$-invariant central generating sets exist by Lemma \ref{lem:daginvariantcentral}.
Then
\[
\nablar(\omega):=\frac{1}{2}\left(\sum_{i} \omega_{i}\otimes \mathrm{d}\pairing{\eta_{i}}{\omega}+\eta_{i}\otimes \mathrm{d}\pairing{{\omega_{i}}}{{\omega}}\right),
\]
is an Hermitian $\sigma$-$\dag$-bimodule connection on $\Omega^{1}_{\dee}(\A)$. 
\end{proposition}
\begin{proof} 
The connection $\nablar$ is Hermitian since for all $\omega,\eta\in\Omega^{1}_{\dee}$ we have
\begin{align*}
\pairing{\omega}{\nablar(\eta)}-\pairing{\nablar(\omega)}{\eta}
&=\frac{1}{2}\left(\sum_{i} \pairing{\omega}{\omega_{i}}\mathrm{d}\pairing{\eta_{i}}{\eta}+\pairing{\omega}{\eta_{i}}\mathrm{d}\pairing{{\omega_{i}}}{{\eta}}+(\mathrm{d}\pairing{\omega}{\eta_{i}})\pairing{\omega_{i}}{\eta}+ (\mathrm{d}\pairing{\omega}{\omega_{i}})\pairing{\eta_{i}}{\eta}\right)\\
&=\mathrm{d}\pairing{\omega}{\eta}.
\end{align*}
By centrality of the $\omega_{i},\eta_{i}$ and $g\circ\sigma=g$ we have
\begin{align*}
\sigma\circ\nablar(\omega)=\frac{1}{2}\left(\sum_{i} \mathrm{d}\pairing{\eta_{i}}{\omega}\otimes \omega_{i} +\mathrm{d}\pairing{{\omega_{i}}}{{\omega}}\otimes \eta_{i} \right)=\frac{1}{2}\left(\sum_{i} \mathrm{d}\pairing{\omega^{\dag}}{\eta_{i}^{\dag}}\otimes \omega_{i} +\mathrm{d}\pairing{\omega^{\dag}}{\omega_{i}^{\dag}}\otimes \eta_{i} \right),
\end{align*}
whereas, using that $\sum \omega_{i}\otimes\mathrm{d}\pairing{\eta_{i}}{\omega}=\sum \omega_{i}^{\dag}\otimes\mathrm{d}\pairing{\eta_{i}^{\dag}}{\omega}$, we find
\begin{align*}
-\dag\circ\nablar\circ\dag(\omega)&=-\frac{1}{2}\left(\sum \omega_{i}\otimes\mathrm{d}\pairing{\eta_{i}}{\omega^{\dag}}+\eta_{i}\otimes\mathrm{d}\pairing{\omega_{i}}{\omega^{\dag}}\right)^{\dag}\\
&=\frac{1}{2}\sum \mathrm{d}\pairing{\omega^{\dag}}{\eta_{i}} \otimes \omega_{i}^{\dag}+\mathrm{d}\pairing{\omega^{\dag}}{\omega_{i}}\otimes \eta_{i}^{\dag}\\
&=\frac{1}{2}\sum \mathrm{d}\pairing{\omega^{\dag}}{\eta_{i}^{\dag}} \otimes \omega_{i}+\mathrm{d}\pairing{\omega^{\dag}}{\omega_{i}^{\dag}}\otimes \eta_{i}
\end{align*}
Thus $\nablar$ is a $\sigma$-$\dag$-bimodule connection.
\end{proof}
Recall from Proposition \ref{prop:alpha} that strong non-degeneracy ensures that we have identifications
\begin{align*}&\alphar:T^{3}_{\mathrm{d}}\to \overrightarrow{\mathrm{Hom}}_{\A}(\Omega^{1}_{\dee},T^{2}_{\dee}),\quad \alphar(\eta\otimes\rho\otimes \tau)(\omega)=\eta\otimes\rho\pairing{\tau^{\dag}}{\omega}\\
&\alphal:T^{3}_{\mathrm{d}}\to \overleftarrow{\mathrm{Hom}}_{\A}(\Omega^{1}_{\dee},T^{2}_{\dee}),\quad \alphal(\eta\otimes\rho\otimes \tau)(\omega)=\pairing{\omega^{\dag}}{\eta}\rho\otimes\tau.
\end{align*}
Using $\alphar$, for instance, we can identify a 1-form-valued endomorphism $A\in \overrightarrow{\mathrm{Hom}}_{\A}(\Omega^{1}_{\dee},T^{2}_{\dee})$ on 
$\Omega^{1}_{\dee}$ with a three-tensor $\alphar^{-1}(A)\in T^{3}_{\mathrm{d}}$.

\begin{theorem}
\label{thm: existence and uniqueness on centred bimods}
Let $(\Omega^{1}_{\dee},\dag,\pairing{\cdot}{\cdot},\sigma^{\mathrm{can}})$ be a centred Hermitian calculus, and define the bimodule projections $P=\Psi\otimes 1,\,Q=1\otimes\Psi: T^{3}_{\dee}(\A)\to  T^{3}_{\dee}(\A)$. Then there exists a unique Hermitian, torsion-free $\sigma$-$\dag$-bimodule connection $\nablar^{G}$ on $\Omega^{1}_{\dee}(\A)$. Moreover, for any pair of finite central $\dag$-invariant generating sequences $\{(\omega_{i},\eta_{i})\}$ of $\Omega^{1}_{\dee}$ as in Proposition \ref{prop: Grassmann}, $\nablar^{G}$ is given by the formula
\begin{equation}
\label{eq:explicit-formula}
\nablar^{G}(\omega):=\frac{1}{2}\left(\sum_{i} \omega_{i}\otimes \mathrm{d}\pairing{\eta_{i}}{\omega}+\eta_{i}\otimes \mathrm{d}\pairing{{\omega_{i}}}{{\omega}}\right)-\alphar((1+4PQ)W),
\end{equation}
where $W=\frac{1}{2}\sum \mathrm{d}\omega_{i}\otimes \eta_{i}^{\dag}+\mathrm{d}\eta_{i}\otimes \omega_{i}^{\dag}\in T^{3}_{\dee}(\A)$.
\end{theorem}
\begin{proof} 
By Lemma \ref{lem:daginvariantcentral} of the Appendix, there exist a pair of central generating sequences $\{(\omega_{i},\eta_{i})\}$ satisfying the hypotheses of Proposition \ref{prop: Grassmann}.
By Corollary \ref{cor:unique}, it suffices to prove that the connection \eqref{eq:explicit-formula}
is an Hermitian torsion-free $\sigma$-$\dag$-bimodule connection.
For the projections $P=\Psi\otimes 1$ and $Q=1\otimes\Psi$ we have $2P-1=\sigma\otimes 1$ and $2Q-1=1\otimes \sigma$, and since the $\omega_{i},\eta_{i}$ are central and the sets $\{\omega_{i}\}=\{\omega_{i}^\dag\}$, $\{\eta_{i}\}=\{\eta_{i}^\dag\}$ are $\dag$-invariant, we find
\begin{equation}
(2P-1)(2Q-1)W=\frac{1}{2}\sum \eta_{i}^{\dag}\otimes \mathrm{d}\omega_{i} +\omega_{i}^{\dag}\otimes \mathrm{d}\eta_{i}=\frac{1}{2}\sum \eta_{i}\otimes \mathrm{d}\omega_{i}^{\dag} +\omega_{i}\otimes \mathrm{d}\eta_{i}^{\dag}=W^{\dag}.
\label{eq:w-wdag}
\end{equation}
Since $PW=QW^{\dag}=0$, applying $Q$ to \eqref{eq:w-wdag} yields
\[
W^{\dag}=(2P-1)(2Q-1)W=(4PQ-2Q+1)W,\quad (4QPQ-Q)W=0.
\]
Therefore, using these relations we find
\begin{align*}
((1+4PQ)W)^{\dag}&=(1+4QP)W^{\dag}=(1+4QP)(4PQ-2Q+1)W\\
&=(8QPQ+4PQ-2Q+1)W=(1+4PQ)W,
\end{align*}
which, along with Proposition \ref{prop: Grassmann}, proves the connection is Hermitian. For torsion-free, we first compute
\begin{align*}
(1-\Psi)\alphar((1+4PQ)W)=\alphar((1-P)(1+4PQ)W)=\alphar((1-P)W)=\alphar(W),
\end{align*}
and then
\begin{align*}
(1-\Psi)\nablar^{G}(\omega)
&=\frac{1}{2}\sum(1-\Psi)(\omega_{i}\otimes \mathrm{d}\pairing{\eta_{i}}{\omega}+\eta_{i}\otimes \mathrm{d}\pairing{\omega_{i}}{\omega})-\alphar(W)(\omega)\\
&=\frac{1}{2}\sum(1-\Psi)(\omega_{i}\otimes \mathrm{d}\pairing{\eta_{i}}{\omega}+\eta_{i}\otimes \mathrm{d}\pairing{\omega_{i}}{\omega})-\mathrm{d}\omega_{i}\pairing{\eta_{i}}{\omega}-\mathrm{d}\eta_{i}\pairing{\omega_{i}}{\omega}
=-\mathrm{d}\omega.
\end{align*}
Lastly, to see that $\nablar^G$ is a bimodule connection, Proposition \ref{prop: Grassmann} says that we need only show that
\[
\sigma\circ\alphar((1+4PQ)W)=-\alphal((1+4QP)W^{\dag}).
\]
Now $\sigma=2\Psi-1$ and $\sigma\circ\alphar=\alphal((\sigma\otimes 1)(1\otimes \sigma)(\sigma\otimes 1))$ and we observe that
\begin{align*}
(2P-1)(2Q-1)(2P-1)(1+4PQ)W&=(2P-1)(2Q-1)(4PQ-1)W\\
&=(4PQ-2P-2Q+1)(4PQ-1)W\\
&=(-4PQ-1)W=-(1+4QP)W^{\dag}.
\end{align*} 
Thus we have
\begin{align*}
\sigma\circ\alphar((1+4PQ)W))&=(2\Psi-1)\alphar((1+4PQ)W)\\
&=\alphal((2P-1)(2Q-1)(2P-1)(1+4PQ)W)\\
&=-\alphal((1+4QP)W^{\dag}).
\end{align*}
Uniqueness follows from Corollary \ref{cor:unique}, or via  \cite[Theorem 5.14]{MRLC} as in Remark \ref{rmk:MR-unique}.
\end{proof}

\begin{corollary}
\label{cor:existence-on-vector-fields}
Let $(\Omega^{1}_{\dee},\dag,\pairing{\cdot}{\cdot},\sigma^{\mathrm{can}})$ be a centred Hermitian calculus. Then $(\LA{\cX}(\cA), \bra{\cdot}\ket{\cdot}_\cA, D_\ph, \LA{\nabla})$ defines a right real pseudo-Riemannian calculus and $(\RA{\cX}(\cA), \prescript{}{\A}{\bra{\cdot}\ket{\cdot}}, D_\ph, \RA{\nabla})$ defines a left real pseudo-Riemannian calculus, and the inner products are strongly non-degenerate.
\end{corollary}

\subsection{Comparing curvatures}
\label{subsec:curvatures}
In this section we compare the definitions of curvature tensor presented in \cite{AWcurvature} and \cite{MRWeitzenbock}. To recall the definition from \cite{MRWeitzenbock},
let $(\Omega^{1}_{\dee},\dag,\pairing{\cdot}{\cdot}_\A,\Psi)$ be an Hermitian second order differential structure. 
Given a right connection $\nablar:\Omega^1_\dee(\A)\to T^2_\dee(\A)$, we can define the curvature $R^{\nablar}:\Omega^1_\dee(\A)\to \Omega^1_\dee(\A)\ox_\A\Lambda^2_\dee(\A)$ (where $\Lambda^2_\dee(\A)=(1-\Psi)T^2_\dee(\A)$) by
\[
R^{\nablar}(\omega)=1\ox(1-\Psi)\circ (\nablar\ox 1+1\ox\d)\circ\nablar(\omega),\qquad \omega\in\Omega^1_\dee(\A).
\]
Similarly, for a left connection $\nablal:\Omega^1_\dee(\A)\to T^2_\dee(\A)$ we can define 
$R^{\nablal}:\Omega^1_\dee(\A)\to \Lambda^2_\dee(\A)\ox_\A\Omega^1_\dee(\A)$ by
\[
R^{\nablal}(\omega)=(1-\Psi)\ox1\circ (1\ox\nablar-\d\ox1)\circ\nablal(\omega),\qquad \omega\in\Omega^1_\dee(\A).
\]
To compare this definition of curvature to that of \cite{AWcurvature}, we require the ability to evaluate, via duality, on three vector fields. This seems difficult in general, but if two of the vector fields are bilinear, we can do the following.

For bilinear vector fields $Z_1,  Z_2 \in D$, a right linear vector field $X \in \RA{\cX}(\cA)$, and a left linear vector field $Y \in \LA{\cX}(\cA)$, we can extend the pairing of Definition \ref{def: product vector fields} of pairs of vector fields with two-tensors to a pairing with three-tensors by defining pairings on a simple three-tensor $\omega\ox\rho\ox\eta$ 
\[
	\begin{split}
		&  \product{X}{\omega\ox\rho\ox\eta}{Z_2 \cdot Z_1} :=  \Pairing{X}{\omega} \Pairing{  Z_2}{\rho} \Pairing{ Z_1}{\eta}\\
		&  \product{Z_1 \cdot Z_2}{\omega\ox\rho\ox\eta}{Y} :=\Pairing{Z_1 }{\omega}\Pairing{ Z_2 }{\rho}\Pairing{\eta}{ Y}
	\end{split}
	\]
and extending by linearity.

\begin{theorem}
\label{thm:curv}
	Let $(\Omega_{\dd}^1(\cA), \dagger, \Psi, \bra{\cdot}\ket{\cdot},\sigma^\mathrm{can})$ be a centred Hermitian differential calculus and $(\nablar, \nablal)$ a pair of conjugate connections. The curvature maps $\RA{R},\LA{R}$ of the affine connections $\partial \mapsto \RA{\nabla}_\partial$ and $\partial \mapsto \LA{\nabla}_\partial$ on vector fields are related to the curvature tensors $R^{\RA{\nabla}}, R^{\LA{\nabla}}$ of $(\RA{\nabla}, \LA{\nabla})$ by 
	\[
	\begin{split}
		& \Pairing{\RA{R}(\partial_1, \partial_2)X}{\omega} = 2 \product{X}{R^{\RA{\nabla}}(\omega)}{\ph(\partial_2) \cdot \ph(\partial_1)} \\
		& \Pairing{\omega}{\LA{R}(\partial_1, \partial_2)Y} = 2 \product{ \ph(\partial_1) \cdot \ph(\partial_2) }{R^{\LA{\nabla}}(\omega)}{Y},
	\end{split}
	\]
	for all $\partial_1,  \partial_2 \in D$, $X \in \RA{\cX}(\cA)$, $Y \in \LA{\cX}(\cA)$ and $\omega \in \Omega_{\dd}^1(\cA)$.
\end{theorem}

\begin{proof}
	We start by proving the formula for $\RA{R}$. Let $X \in \RA{\cX}(\cA)$, and $\omega\in \Omega^1_\dee(\A)$. Then by definition
	\[
		\Pairing{\RA{R}(\partial_1, \partial_2)X}{\omega} = \Pairing{\RA{\nabla}_{\partial_1} \RA{\nabla}_{\partial_2}  X}{\omega} - \Pairing{\RA{\nabla}_{\partial_2} \RA{\nabla}_{\partial_1}  X}{\omega} - \Pairing{\RA{\nabla}_{[\partial_1, \partial_2]}X}{\omega}.
	\]
	Using the definition of the right affine connection,
	\begin{alignat*}{2}
		& \Pairing{\RA{\nabla}_{\partial_1} \RA{\nabla}_{\partial_2}  X}{\omega} 
		&& = \partial_1 \Pairing{\RA{\nabla}_{\partial_2}X}{\omega} - \product{ \RA{\nabla}_{\partial_2} X }{\RA{\nabla}(\omega)}{\ph(\partial_1)} \\
		& && = \partial_1 \partial_2 \Pairing{X}{\omega} - \partial_1 \product{X}{\RA{\nabla}(\omega)}{\ph(\partial_2)} - \product{ \RA{\nabla}_{\partial_2} X }{\RA{\nabla}(\omega)}{\ph(\partial_1)}, \\
		& \Pairing{\RA{\nabla}_{\partial_2} \RA{\nabla}_{\partial_1}  X}{\omega} && = \partial_2 \partial_1 \Pairing{X}{\omega} - \partial_2 \product{X}{\RA{\nabla}(\omega)}{\ph(\partial_1)} - \product{ \RA{\nabla}_{\partial_1} X }{\RA{\nabla}(\omega)}{\ph(\partial_2)} \\
		& \Pairing{\nablar_{[\partial_1, \partial_2]}X}{\omega} && = \partial_1 \partial_2 \Pairing{X}{\omega} - \partial_2 \partial_1 \Pairing{X}{\omega} - \product{X}{\RA{\nabla}(\omega)}{ \ph \big( [\partial_1, \partial_2] \big) }.
	\end{alignat*}
	Hence the curvature becomes
	\begin{alignat*}{2}
		\Pairing{\RA{R}(\partial_1, \partial_2)X}{\omega} & =\ && \product{X}{\RA{\nabla}(\omega)}{ \ph \big( [\partial_1, \partial_2] \big) } \\
		& &&\ - \partial_1 \product{X}{\RA{\nabla}(\omega)}{\ph(\partial_2)} + \partial_2 \product{X}{\RA{\nabla}(\omega)}{\ph(\partial_1)} \\
		& &&\ - \product{\nablar_{\partial_2} X}{\RA{\nabla}(\omega)}{\ph(\partial_1)}  + \product{\nablar_{\partial_1} X}{\RA{\nabla}(\omega)}{\ph(\partial_2)}.
	\end{alignat*}
	We write $\RA{\nabla}(\omega) = \eta^i \otimes \rho^i$ and $\rho^i = a^i_j \dd b^i_j$, with implicit sums over $i$ and $j$. Then
	\[
		\partial_2 \product{X}{\RA{\nabla}(\omega)}{\ph(\partial_1)} = \big( \partial_2 \Pairing{X}{\eta^i} \big) \Pairing{\ph(\partial_1)}{\rho^i} + \Pairing{X}{\eta^i} \big( \partial_2 \Pairing{\ph(\partial_1)}{\rho^i} \big)
	\]
	and
	\[
		\partial_2 \Pairing{\ph(\partial_1)}{\rho^i} = \partial_2 \big( a^i_j \partial_1 b^i_j \big) = ( \partial_2 a^i_j ) ( \partial_1 b^i_j ) + a^i_j \partial_2 \partial_1 b^i_j,
	\]
	which yield
	\begin{align}
	     \partial_2 \product{X}{\RA{\nabla}(\omega)}{\ph(\partial_1)} &= \big( \partial_2 \Pairing{X}{\eta^i} \big) \Pairing{\ph(\partial_1)}{\rho^i}
	     + \Pairing{X}{\eta^i} \product{\ph(\partial_2)}{\dd a^i_j \ox \dd b^i_j}{\ph(\partial_1)} \nonumber\\
	    &+ \Pairing{X}{\eta^i} (a^i_j \partial_2 \partial_1 b^i_j), \label{eq: curv_comp_1.1}\\
	    \partial_1 \product{X}{\RA{\nabla}(\omega)}{\ph(\partial_2)} &= \big( \partial_1 \Pairing{X}{\eta^i} \big) \Pairing{\ph(\partial_2)}{\rho^i}
	     +  \Pairing{X}{\eta^i} \product{\ph(\partial_1)}{\dd a^i_j \ox \dd b^i_j}{\ph(\partial_2)} \nonumber \\
	    &+  \Pairing{X}{\eta^i} (a^i_j \partial_1 \partial_2 b^i_j),  \label{eq: curv_comp_1.2} \\
	    \product{X}{\RA{\nabla}(\omega)}{\ph \big( [\partial_1, \partial_2] \big)} &= \Pairing{X}{\eta^i} (a^i_j \partial_2 \partial_1 b^i_j) - \Pairing{X}{\eta^i} (a^i_j \partial_1 \partial_2 b^i_j). \label{eq: curv_comp_1.3}
	\end{align}
	Furthermore, using the definition of vector field products
	\begin{align}
	    \label{eq: curv_comp_2.1} \product{\RA{\nabla}_{\partial_1} X}{\RA{\nabla}(\omega)}{\ph(\partial_2)} = \big[ \partial_1 \Pairing{X}{\eta^i} - \product{X}{\RA{\nabla}(\eta^i)}{\ph(\partial_1)} \big] \Pairing{\ph(\partial_2)}{\rho^i}, \\
        \label{eq: curv_comp_2.2} \product{\RA{\nabla}_{\partial_2} X}{\RA{\nabla}(\omega)}{\ph(\partial_1)} = \big[ \partial_2 \Pairing{X}{\eta^i} - \product{X}{\RA{\nabla}(\eta^i)}{\ph(\partial_2)} \big] \Pairing{\ph(\partial_1)}{\rho^i}.
	\end{align}
	Computing \eqref{eq: curv_comp_1.1} - \eqref{eq: curv_comp_1.2} + \eqref{eq: curv_comp_1.3} + \eqref{eq: curv_comp_2.1} - \eqref{eq: curv_comp_2.2} and simplifying yields
	\[
	    \Pairing{\RA{R}(\partial_1, \partial_2)X}{\omega} = \product{X}{\RA{\nabla}(\eta^i) \otimes \rho^i + \eta^i \otimes \dd a^i_j \otimes \dd b^i_j}{\ph(\partial_2) \cdot \ph(\partial_1) - \ph(\partial_1) \cdot \ph(\partial_2)}.
	\]
	We can then use Lemma \ref{lemma: antisym_vector_fields} and conclude that
	\[
	    \Pairing{\RA{R}(\partial_1, \partial_2)X}{\omega} = 2 \product{X}{1 \otimes (1 - \Psi) \circ (\RA{\nabla} \otimes 1 + 1 \otimes \dd_\Psi) \circ \RA{\nabla} (\omega)}{\ph(\partial_2) \cdot \ph(\partial_1)}.
	\]
	Since we have chosen a representation $\RA{\nabla}(\omega) = \eta^i \otimes \rho^i$, we need to check that the end result does not depend on this choice. Indeed, the curvature tensor $R^{\RA{\nabla}}(\omega)$ is well-defined on $\Omega_{\dd}^1(\cA) \otimes_\cA \Lambda_{\dd}^2(\cA)$, so our computation does not depend on the choice of representative.
	
	The second formula of the theorem follows from the first by taking the dagger
	\[
	\begin{split}
	    \Pairing{\eta^\dagger}{\big( \RA{R}(\partial_1, \partial_2) X \big)^\dagger} &= \Pairing{\RA{R}(\partial_1, \partial_2)X}{\eta}^{*} = 2 \product{X}{R^{\RA{\nabla}}(\eta)}{\ph(\partial_2) \cdot \ph(\partial_1)}^{*} \\
	    &= 2 \product{\ph(\partial_1) \cdot \ph(\partial_2)}{\big( R^{\RA{\nabla}}(\eta) \big)^\dagger}{X^\dagger}.
    \end{split}
    \]
    We know that $\big( R^{\RA{\nabla}}(\omega) \big)^\dagger = R^{\LA{\nabla}}(\omega^\dagger)$ for all $\omega \in \Omega^1_\dee(\A)$ (see \cite[Section 3.1]{MRWeitzenbock}) and moreover
    \[
        \dagger \circ \RA{R}(\partial_1, \partial_2) = \dagger \circ \big( \RA{\nabla}_{\partial_1} \RA{\nabla}_{\partial_2} - \RA{\nabla}_{\partial_2} \RA{\nabla}_{\partial_1} - \RA{\nabla}_{[\partial_1, \partial_2]} \big) = \LA{R}(\partial_1, \partial_2) \circ \dagger
    \]
    since $\dagger \circ \RA{\nabla}_\partial = \LA{\nabla}_\partial \circ \dagger$ (see Proposition \ref{prop: dagger and affine connections}). Therefore setting $\omega := \eta^\dagger$ and $Y := X^\dagger$, we find that
    \[
        \Pairing{\omega}{\LA{R}(\partial_1, \partial_2)Y} = 2 \product{ \ph(\partial_1) \cdot \ph(\partial_2) }{R^{\LA{\nabla}}(\omega)}{Y}. \qedhere
    \]
\end{proof}

\subsection{Isospectral deformations}
\label{subsec:isospectral}

We now briefly show that, using our approach, we can recover the pseudo-Riemannian calculi constructed in \cite{AWcurvature} on the noncommutative 2-torus $\T^2_\theta$ and noncommutative 3-sphere $\SS^3_\theta$, from general results on $\theta$-deformed spectral triples in \cite{Gosw&alCentredBimod} and \cite{MRLC}. These manifolds possess the property that the isometric $\T^2$-action on $M$ used to define $\theta$-deformations is a free action. This is crucial for the proof given in \cite{Gosw&alCentredBimod}, because free torus actions yield centred bimodules for the deformed algebra, and likewise for the proof of Theorem \ref{thm: existence and uniqueness on centred bimods}.

For a compact Riemannian manifold $(M,g)$ equipped with a Dirac bundle $\slashed{S}\to M$, we have an associated spectral triple $(C^\infty(M),L^2(M,\slashed{S}),\slashed{D})$. Then $\Omega^{1}_{\slashed{D}}(C^\infty(M))\cong \Omega^1(M)\ox\C$ \cite[Chapter VI]{BRB}, and we let $\pairing{\cdot}{\cdot}_{g}$ be the inner product on $\Omega^{1}_{\slashed{D}}(C^\infty(M))$ induced by $g$. The junk two-tensors in $T^{2}_{\slashed{D}}(C^{\infty}(M))\simeq T^{2}(M)$ coincide with the module of symmetric tensors \cite[Example 4.26]{MRLC}. Thus for $\sigma: T^{2}_{\slashed{D}}(M)\to T^{2}_{\slashed{D}}(M)$ the standard flip map we can set $\Psi:=\frac{1+\sigma}{2}$ for the junk projection.

\begin{theorem} 
\label{thm:LC-class}
Let $(M,g)$ be a compact Riemannian manifold with a Dirac bundle $\slashed{S}\to M$.  Then $(\Omega^{1}_{\slashed{D}}(C^{\infty}(M)),\dag, \Psi,\pairing{\cdot}{\cdot}_{g})$ is a centred Hermitian differential calculus and there exists a unique Hermitian torsion-free $\dag$-bimodule connection $(\nablar^{G},\sigma)$ on $\Omega^{1}_{\slashed{D}}(C^\infty(M))\cong \Omega^1(M)\ox\C$. The restriction to real forms $\nablar^{G}:\Omega^{1}(M) \to \Omega^{1}(M) \ox \Omega^1(M)$ coincides with the (classical) Levi-Civita connection on $\Omega^{1}(M)$.
\end{theorem}

Applying $\theta$-deformation for a free torus action yields a noncommutative algebra 
$C^\infty(M_\theta)$ for which the data $(C^\infty(M_\theta), L^2(M, \slashed{S}), \slashed{D})$ is still a spectral triple, and the associated one-forms $\Omega_{\slashed{D}}^1(M_\theta)$ are a centred bimodule.

\begin{theorem}[\cite{Gosw&alCentredBimod}]
    Let $M$ be a compact spin Riemannian manifold with a free isometric torus action. Let $\slashed{S}$ be a Dirac bundle on $M$ and $\slashed{D}$ be the associated Dirac operator. Then the bimodule $\Omega_{\slashed{D}}^1(M_\theta)$ of differential one-forms over the $\theta$-deformed spectral triple $(C^\infty(M_\theta), L^2(M, \slashed{S}), \slashed{D})$ is a centred bimodule.
\end{theorem}

Thus we can use the correspondences between vector fields and forms that we have exhibited in this paper in the case of $\theta$-deformations coming from free torus actions, in particular, the 2-torus and 3-sphere. All we need is the existence of Levi-Civita connections on these $\theta$-deformed spectral triples. More specifically, we need the existence of a $\dagger$-bimodule connection for the canonical braiding $\sigma$ on our centred bimodule of differential forms. This is proved as \cite[Theorem 5.4]{Gosw&alCentredBimod}.

\begin{theorem}[\cite{Gosw&alCentredBimod}]
    Let $M$ be a compact spin Riemannian manifold with a free isometric torus action. Then there exists a unique Hermitian and torsion-free $\dagger$-bimodule connection $(\RA{\nabla}_\theta, \sigma)$ on the $\theta$-deformed spectral triple $(C^\infty(M_\theta), L^2(M, \slashed{S}), \slashed{D})$.
\end{theorem}

This result is also a consequence of Theorem 6.12 in \cite{MRLC}, which guarantees the existence and uniqueness of the Hermitian and torsion-free connection for a general $\theta$-deformed spectral triple (even when the torus action is not necessarily free). To see this, we simply note that the braiding $\sigma_\theta$ considered in \cite[Section 6.2]{MRLC} coincides with the canonical braiding when $\Omega^1_{\slashed{D}}(M_\theta)$ is centred. Indeed, if $\omega$ and $\eta$ are central homogeneous forms in $\Omega^1_{\slashed{D}}(M_\theta)$ then $\sigma_\theta(\omega \otimes \eta) = \Theta(\omega, \eta) \eta \otimes \omega = \eta \otimes \omega$. We also note that \cite{MRLC} proves the identity $\Psi_\theta = \frac{1}{2}(1 + \sigma_\theta)$ for general $\theta$-deformations.

All the assumptions of Theorem \ref{thm: pseudo-Riemannian_calculus} are thus satisfied by $\theta$-deformed spectral triples coming from free toral actions. Hence we can state the following result.

\begin{corollary}
\label{cor:AW-fini}
    Let $M$ be a compact spin Riemannian manifold with a free isometric torus action. Then $(\LA{\cX}(C^\infty(M_\theta)), \bra{\cdot}\ket{\cdot}_{C^\infty(M_\theta)}, D_\ph, \LA{\nabla})$ is a (right) real pseudo-Riemannian calculus with strongly non-degenerate inner product, where $\LA{\nabla}$ is the affine connection induced by the unique (left) Levi-Civita connection on $\Omega_{\slashed{D}}^1(M_\theta)$.
\end{corollary}


\appendix
\section{Finitely generated projective inner product $\dag$-bimodules}
\label{sec:fgpipdb}
In this appendix we collect some general results for abstract finitely generated projective strongly non-degenerate inner product bimodules.
\begin{definition}
\label{defn:stern-bimod-app}
A \emph{$\dag$-bimodule} over the $*$-algebra $\A$ is an $\A$-bimodule $\mathcal{X}$ that is equipped with  an antilinear involution $\dag:\mathcal{X}\to\mathcal{X}$ such that $(axb)^{\dag}=b^{*}x^{\dag}a^{*}$. An \emph{inner product $\dag$-bimodule} is a $\dag$-bimodule $\X$ equipped with a pairing
\[
\pairing{\cdot}{\cdot}_\A:\X\times\X\to \A,\quad (x,y)\mapsto \pairing{x}{y}_\A,
\]
such that
\begin{enumerate}
\item $\pairing{x}{y a}_\A=\pairing{x}{y}_\A a$;
\item $\pairing{x}{y}^{*}_\A=\pairing{y}{x}_\A$;
\item $\pairing{ax}{y}=\pairing{x}{a^{*}y}$;
\item $\pairing{\cdot}{\cdot}$ is \emph{strongly non-degenerate}: the map $x\mapsto \left( y \mapsto \pairing{x^\dagger}{y}_\cA \right)$ is an isomorphism $\X\to \overrightarrow{\mathrm{Hom}}_{\A}(\X,\A)$ of left $\A$-modules.
\end{enumerate}
\end{definition}
\begin{remark}
A strongly non-degenerate inner product is weakly non-degenerate, in the sense that the map $x\mapsto (y\mapsto \pairing{x}{y}_\A)$ is injective.
\end{remark}
We will omit the subscript $\A$ on inner products below.

\begin{lemma} 
\label{lem:generatingsequences}
Let $\X$ be a finitely generated projective right module with a strongly non-degenerate inner product $\pairing{\cdot}{\cdot}$ (satisfying 1,2 and 4 of Definition \ref{defn:stern-bimod-app}). Then there exist a finite set of pairs $\{(x_{i},x'_{i})\}_{i=1}^{n}\subset\X\times\X$ such that
\[x=\sum_{i=1}^{n}x_{i}\pairing{x'_{i}}{x}=\sum_{i=1}^{n}x'_{i}\pairing{x_{i}}{x}\]
 for all $x\in\X$. We call $\{(x_{i},x'_{i})\}_{i=1}^{n}\subset\X\times\X$ a generating pair.
\end{lemma}
\begin{proof}
Choose a finite generating set $\{x_{i}\}$ for $\X$, that is for each $x\in\X$ there exist $a_{i}\in \A$ such that $x=\sum_{i} x_{i}a_{i}.$
Consider the surjective bimodule map
\[\pi:\A^{n}\to \X,\quad (a_{i})\mapsto \sum_{i} x_{i}a_{i}.\]
Since $\X$ is projective there is a splitting $s:\X\to \A^{n}$, $s(x)=(s_{i}(x))$, with $s_{i}\in \textrm{Hom}(\X,\A)$. Now since $\X$ carries a strongly non-degenerate inner product, there exist $x'_{i}\in \X$ such that $s_{i}(x)=\pairing{x'_{i}}{x}$ and thus
\[
x=\sum_{i} x_{i}\pairing{x'_{i}}{x}.
\]
For every $z$ we have
\[\pairing{\sum_{i}x'_{i}\pairing{x_{i}}{x}}{z}=\sum_{i}\pairing{x}{x_{i}}\pairing{x'_{i}}{z}=\pairing{x}{\sum_{i} x_{i}\pairing{x'_{i}}{z}}=\pairing{x}{z},\]
so that by non-degeneracy $x=\sum_{i}x'_{i}\pairing{x_{i}}{x}$ as well. 
\end{proof}
\begin{remark} If $\A$ is dense and spectral invariant inside a $C^{*}$-algebra $A$, then any positive definite (hence \emph{weakly} non-degenerate) finitely generated projective right inner product module admits a \emph{frame} \cite{FL02, MRLC}. This a set $\{x_i\}$ such that $\{(x_i,x_i)\}$ satisfies the conclusion of Lemma \ref{lem:generatingsequences} (that is, $x_i=x_{i}'$). The proof of existence of such a frame requires some analysis, afforded by the spectral invariance assumption. Existence of frames, in turn, \emph{implies} strong non-degeneracy of the inner product.
\end{remark}
\begin{remark} 
\label{rem:algebraicsequences}
Given a pair of sequences $\{(x_{i},x_{i}')\}$ as in Lemma \ref{lem:generatingsequences}, the two tensor $\sum_{i}x_{i}\otimes x_{i}'$ is a generalised quantum metric relative to the inner product $\pairing{\cdot}{\cdot}$ in the sense of \cite[Definition 1.5]{BMbook}. Such sequences are an algebraic analogue of the notion of frame in a Hilbert $C^{*}$-module, and have appeared in the algebra literature, for instance \cite{BMconnections}, and recently in \cite{AH25}.
\end{remark}
\begin{corollary} Let $\X$ and $\Y$ be finitely generated projective strongly non-degenerate right inner product modules. For every $T\in\mathrm{Hom}_{\A}(\X,\Y)$ there exists $T^{*}\in\mathrm{Hom}_{\A}(\Y,\X)$ such that for $x\in\X$ and $y\in\Y$ we have $\pairing{Tx}{y}=\pairing{x}{T^{*}y}$.
\end{corollary}
\begin{proof}
Let $\{(x_{i},x_{i}')\}$ be a generating pair for $\X$ and set $T^{*}y:=\sum_{i}x_{i}\pairing{Tx_{i}'}{y}$. Then
\begin{align*}
\pairing{x}{T^{*}y}=\sum_{i}\pairing{x}{x_{i}}\pairing{Tx_{i}'}{y}=\sum_{i}\pairing{Tx_{i}'\pairing{x_{i}}{x}}{y}=\pairing{Tx}{y},
\end{align*}
as claimed.
\end{proof}
All tensor powers $\X^{\otimes m}$ carry right and left inner products, defined inductively via
\[\pairing{x_{1}\otimes y_{1}}{x_{2}\otimes y_{2}}:=\pairing{y_{1}}{\pairing{x_{1}}{x_{2}}y_{2}}.\]
Using these inner products we obtain bimodule maps
\begin{align}
\alphar&:\X^{\otimes(n+k)}\to \overrightarrow{\textnormal{Hom}}_{\A}(\X^{\otimes k},\X^{\otimes n}), \quad\alphar(x\otimes y)(z):=x\pairing{y^{\dag}}{z}_\A\nonumber\\
\alphal&:\X^{\otimes(n+k)}\to \overleftarrow{\textnormal{Hom}}_{\A}(
\X^{\otimes k},\X^{\otimes n}), \quad\alphal( x\otimes y)(z):={}_\A\pairing{z}{y^\dag}x,
\label{eq:alphas}
\end{align}
where $y,z\in \X^{\otimes k}$ and $x\in \X^{\otimes n}$. The following is an extension of Proposition \ref{prop:musicalisos} and should be viewed as abstract raising and lowering indices on tensors via the musical isomorphisms.
\begin{proposition}
\label{prop:alpha}
Let $\X$ be a finitely generated projective $\dag$-bimodule with a strongly non-degenerate inner product. Then for all $n,k\geq 0$, the maps $\alphar$ and $\alphal$ are isomorphisms. In particular, the inner product on each $\X^{\otimes k}$ is strongly non-degenerate.
\end{proposition}

\begin{proof} Given a generating pair $\{(x_{i},x'_{i})\}_{i\in I}$ for $\X$ indexed by the finite set $I$. Consider the set
\begin{equation}
\label{eq:productgeneratingsequence}
\{(x_{i_{1}}\otimes \cdots \otimes x_{i_{k}}, x'_{i_{1}}\otimes\cdots\otimes x'_{i_{k}})\}_{j=(i_{1},\cdots,i_{k})\in I^{k}}\end{equation}
and any elementary tensor $\xi=\xi_{1}\otimes\cdots\otimes \xi_{k}$ and compute
\begin{align*}
\sum_{j\in I^{k}} x_{j}\pairing{x'_{j}}{\xi}=\sum_{j\in I^{k-1}}\sum_{n\in I}x_{j}\otimes x_{n}\pairing{x'_{n}}{\pairing{x'_{j}}{\xi_1 \ox \cdots \ox \xi_{k-1} } \xi_k} = \sum_{j\in I^{k-1}}x_{j}\otimes\pairing{x'_{j}}{\xi_1 \ox \cdots \ox \xi_{k-1}}\xi_k.
\end{align*}
Now employ induction to conclude that \eqref{eq:productgeneratingsequence} is a pair of generating sequences for $\X^{\otimes k}$. The inverses of the maps $\alphar$ and $\alphal$ are given by
\begin{align*}
\alphar^{-1}(T)=\sum_{j\in I^{k}} Tx_{j}\otimes x'_{j},\quad \alphal^{-1}(T):=\sum_{j\in I^{k}}x'_{j}\otimes Tx_{j},
\end{align*}
which proves the claim.
\end{proof}
We now restrict ourselves further to the case of centred bimodules, that is, modules generated by their centre. Recall that a braiding on a $\dag$-$\cA$-bimodule $\cM$ is an invertible map $\sigma : \cM \ox_\cA \cM \rightarrow \cM \ox_\cA \cM$ such that $\sigma^{-1} \circ \dag = \dag \circ \sigma$. The following result is due to Skeide and we recall it here for convenience.

\begin{proposition}\cite{Skeide}
\label{prop: braiding}
    On a centred $\dagger$-bimodule $\cM$ over $\cA$, there exists a unique braiding $\sigma^\mathrm{can}$ which satisfies $\sigma^{\mathrm{can}}(x \otimes y) = y \otimes x$ whenever $x$ or $y \in \cZ(\cM)$. This braiding is an involution.
\end{proposition}

\begin{proof}
    Uniqueness follows from bilinearity of $\sigma^\mathrm{can}$. Indeed, $\cZ(\X)$ generates $\cM$ so two-tensors of the form $x \otimes y$ with $x$ and $y$ central generate $\X \otimes_\cA \X$ as a bimodule. Existence follows from the fact that the flip map $x \otimes y \mapsto y \otimes x$ is well-defined when $x$ and $y$ are central, and central elements generate $\cM$. If $x$ is central and $y=\sum_{i} e_{i}a_{i}=\sum_{i} a_{i}e_{i}$, with $e_{i}$ central and $a_{i}\in \A$, then 
\[
\sigma^{\mathrm{can}}(x\otimes y)=\sum_{i} a_{i}\sigma(x\otimes e_{i})=\sum_{i} a_{i}e_{i}\otimes x=y\otimes x.
\]    
     Clearly $\sigma^\mathrm{can} = (\sigma^\mathrm{can})^{-1}$ on central forms, so by bi-$\cA$-linearity $\sigma^\mathrm{can}$ is an involution on $\cM \otimes_\cA \cM$. It satisfies the braiding condition $\sigma^\mathrm{can} \circ \dagger = \dagger \circ (\sigma^\mathrm{can})^{-1}=\dagger \circ \sigma^\mathrm{can}$.
\end{proof}

In the case of a centred bimodule $\X$ with a symmetric inner product we can improve on Lemma \ref{lem:generatingsequences}.
\begin{lemma}
\label{lem:daginvariantcentral}
Let $\X$ be a centred finitely generated projective $\dag$-bimodule with non-degenerate inner product $\pairing{\cdot}{\cdot}$ such that $g\circ\sigma^{\mathrm{can}}=g$. Then there exists a finite set of pairs $\{(x_{i},x'_{i})\}_{i=1}^{n}\subset\mathcal{Z}(\X)\times\mathcal{Z}(\X)\subset\X\times\X$ of central elements such that $\{(x_{i}^{\dag},x'^{\dag}_{i})\}_{i=1}^{n}=\{(x_{i},x'_{i})\}_{i=1}^{n}$ and for all $x\in \X$ we have 
\[x=\sum_{i}x_{i}\pairing{x'_{i}}{x}=\sum_{i}x'_{i}\pairing{x_{i}}{x}=\sum_{i} x_{i}^{\dag}\pairing{x'^{\dag}_{i}}{x}=\sum_{i} x'^{\dag}_{i}\pairing{x_{i}^{\dag}}{x}.\]
\end{lemma}
\begin{proof}
Since $\X$ is centred we can find central sequences $\{(x_{i},x'_{i})\}$ satisfying the conclusion of Lemma \ref{lem:generatingsequences}. To obtain $\dag$-invariant sets, choosing $x'_{i}$ as above we have
\begin{align*}
\sum_{i} x_{i}^{\dag}\pairing{x'^{\dag}_{i}}{x}&=\sum_{i}\pairing{x'^{\dag}_{i}}{x}x_{i}^{\dag}=\sum_{i} \pairing{x^{\dag}}{x'_{i}}x_{i}^{\dag}=\left(\sum_{i} x_{i}\pairing{x'_{i}}{x^{\dag}}\right)^{\dag}=x.
\end{align*}
Therefore
\[x=\frac{1}{2}\left(\sum_{i} x_{i}\langle x'_{i},x\rangle+x_{i}{^\dag}\langle x'^{\dag}_{i},x\rangle\right).\]
Thus the set $\{(\frac{x_{i}}{\sqrt{2}},\frac{x_{i}^\dag}{\sqrt{2}}),(\frac{x'_{i}}{\sqrt{2}},\frac{x'^{\dag}_{i}}{\sqrt{2}})\}$ satisfies the conclusion of the Lemma.
\end{proof}


\end{document}